\documentclass[10pt]{amsart}
\usepackage{graphicx,amssymb,amsfonts,amsmath,amsthm,newlfont}
\usepackage{epsfig,url}

\usepackage[all,2cell]{xy} \UseAllTwocells \SilentMatrices

\newtheorem{thm}{Theorem}[section]
\newtheorem{cor}[thm]{Corollary}
\newtheorem{lem}[thm]{Lemma}
\newtheorem{prop}[thm]{Proposition}
\theoremstyle{definition}
\newtheorem{defn}[thm]{Definition}
\newtheorem{conj}{Conjecture} 

\newtheorem{ex}[thm]{Examples}
\newtheorem{example}[thm]{Example}
\theoremstyle{remark}
\newtheorem{rem}[thm]{Remark}
\numberwithin{equation}{section}

\newcommand{\Z}{\mathbb Z}
\newcommand{\C}{\mathbb C}

\newcommand{\p}{\mathfrak{p}}
\newcommand{\R}{\mathbb R}

\newcommand{\Pro}{\mathbb P}

\newcommand{\gr}{\mathrm{gr}}

\font \rus= wncyr10
\newcommand{\sha}{\, \hbox{\rus x} \,}

\newcommand{\Mot}{\mathcal{M}}
\newcommand{\Ho}{\mathcal{H}}

\newcommand{\MT}{\mathcal{MT}}
\newcommand{\GMT}{\mathcal{G}_{MT}}

\newcommand{\zetam}{\zeta^{ \mathfrak{m}}}

\newcommand{\Imot}{I^{\mathfrak{m}}}

\newcommand{\Q}{\mathbb Q}

\newcommand{\U}{\mathcal{U}}

\newcommand{\To}{\longrightarrow}

\newcommand{\A}{\mathbb{A}}
\newcommand{\G}{\mathbb{G}}

\newcommand{\tdots}{.\, .\,}

\newcommand{\tone}{\overset{\rightarrow}{1}\!}
\newcommand{\opi}{{}_0 \Pi_{1}}
\newcommand{\Or}{\mathcal{O}}

\newcommand{\g}{\mathfrak{g}}
\newcommand{\UMT}{\mathcal{U}_{MT}}
\newcommand{\circb}{\, \underline{\circ}\, }

\newcommand{\Amt}{\mathcal{A}^{\MT}}
\newcommand{\dch}{dch}
\newcommand{\Jmt}{J^{\MT} }

\newcommand{\dd}{\mathfrak{D}}

\newcommand{\e}{\mathsf{e}}
\newcommand{\II}{\mathbb{I}}

\newcommand{\godd}{\mathfrak{dg}^{odd}}

\newcommand{\gmzv}{\g^{\mathfrak{m}} }

\newcommand{\gd}{\mathfrak{dg}^{\mathfrak{m}}}

\newcommand{\gdsh}{\mathfrak{ls} }

\newcommand{\Ao}{\mathcal{A} }

\newcommand{\y}{\mathsf{y} }
\newcommand{\Ss}{\mathsf{S} }

\newcommand{\3}{\mathsf{3} }
\newcommand{\5}{\mathsf{5} }
\newcommand{\7}{\mathsf{7} }

\newcommand{\PP}{\mathfrak{p} }

\newcommand{\SL}{\mathrm{SL} }

\newcommand{\Lie}{\mathrm{Lie}}
\newcommand{\Phim}{\Phi^{\mathfrak{m}}}

\newcommand{\dD}{\mathbb{D}}

\newcommand{\Deltas}{\Delta_{\!\sha}}
\newcommand{\Deltaps}{\Delta'_{\!\sha}}
\newcommand{\stu}{*}

\newcommand{\DMD}{\mathfrak{dm}}

\newcommand{\Deltat}{\Delta_{\stu}}

\begin{document}
\author{Francis Brown}
\begin{title}[Depth-graded MZV's]{Depth-graded motivic multiple zeta values}\end{title}
\maketitle
\begin{abstract} 
We study the depth filtration on  multiple zeta values, the motivic Galois group of mixed Tate motives over $\Z$ and the Grothendieck-Teichm\"uller group,  and its relation to modular forms.
Using period polynomials for cusp forms for $\SL_2(\Z)$, we construct an explicit  Lie algebra  of solutions to the linearized double \mbox{shuffle} equations, which gives    a conjectural  description of  all identities  between multiple zeta values  modulo $\zeta(2)$ and modulo lower depth.  We formulate a  single conjecture about the homology of this Lie algebra  which  implies   conjectures due to   Broadhurst-Kreimer,  Racinet, Zagier and Drinfeld on the structure of multiple zeta values and on the Grothendieck-Teichm\"uller Lie algebra. 
\end{abstract}

\section{Introduction}

We begin by motivating the results of this paper from two apparently different, but in fact  equivalent,  perspectives. 
\subsection{Depth filtration on multiple zeta values}
Multiple zeta values  are defined for integers $n_1,\ldots, n_{r-1}\geq 1$ and $n_r\geq2$ by
$$\zeta(n_1,\ldots, n_r) = \sum_{1\leq k_1< \ldots < k_r}  {1 \over k_1^{n_1} \ldots k_r^{n_r} }\ .$$
Their weight is the quantity $n_1+\ldots +n_r$, and their depth is the number of indices $r$. Relations between multiple zeta values of depth two were first studied by Euler. 
Let $\mathcal{Z}_N$ denote the $\Q$-vector space spanned by multiple zeta values in weight $N$.  Zagier conjectured that the dimension of $\mathcal{Z}_N$ can be expressed using   the generating series
\begin{equation}  \label{introdimZnGF} \sum_{N\geq 0}  \dim_{\Q}  \, (\mathcal{Z}_N)  s^N =  {1 \over 1-s^2-s^3}\ . 
\end{equation} 
Using the theory of mixed Tate motives, Goncharov and Terasoma independently showed  that  $\dim_{\Q}  \mathcal{Z}_N$   is bounded above by the coefficient of $s^N$ in the right-hand side of \eqref{introdimZnGF}. Furthermore, if one  replaces $\mathcal{Z}_N$ with the $\Q$-vector space of motivic multiple zeta values $\zetam(n_1,\ldots, n_r)$ of weight $N$, then   equation \eqref{introdimZnGF} is a theorem 
\cite{BrMTZ}. The rational function  on the right-hand side of \eqref{introdimZnGF}  can be interpreted as follows: it is  the Poincar\'e series of the  free module, generated by  $\zetam(2n)$ for $n\geq 1$,  over the graded dual of the universal enveloping algebra of the Lie algebra of the category of mixed Tate motives over $\Z$, which is free with one generator in every odd degree $\leq -3$.

Based on numerical experiments, Broadhurst and Kreimer  formulated a fascinating and more refined  conjecture. Let $\mathcal{Z}_{N,d}$  denote the $\Q$-vector space spanned by multiple zeta values in weight $N$ and depth $d$. They propose  that
\begin{equation} \label{introBK} \sum_{N, d\geq 0} \dim_{\Q}\,   (\mathcal{Z}_{N,d})  s^N t^d = { 1  + \mathbb{E}(s) t \over 1- \mathbb{O}(s) t + \mathbb{S}(s) t^2 - \mathbb{S}(s) t^4}\ ,
\end{equation}
 where, using the notation from  (\cite{IKZ},  appendix),
\begin{equation} \label{EOSdef}
\mathbb{E}(s) = {s^2 \over 1-s^2}   \quad , \quad \mathbb{O}(s) = {s^3 \over 1-s^2}  \quad  , \quad \mathbb{S}(s) = {s^{12} \over (1-s^4)(1-s^6)}  \ .
\end{equation}
Formula \eqref{introBK} specialises to the statement \eqref{introdimZnGF}  upon setting $t=1$.

 The meaning of this conjecture is still mysterious, but one goal of this paper is to offer an interpretation   of \eqref{introBK}.
 The   series  $\mathbb{E}(s)$ and  $\mathbb{O}(s)$ are  the generating series  for the dimensions of the spaces of even and odd single zeta values respectively, and 
 $\mathbb{S}(s)$  is the  generating series for the dimensions of the space of  cusp forms  for the full modular group $\mathrm{SL}_2(\Z)$. The  first prediction of 
 \eqref{introBK}, due to the presence of a non-trivial coefficient of $t^2$ in the denominator of the right-hand side,  is  the existence of  an extra relation between double zeta values of even weight for  every cusp form, modulo multiple zeta values of lower depth (single zeta values).  
 These relations have indeed been shown to exist and are  well-understood by the work of Gangl, Kaneko and Zagier \cite{GKZ}, who exhibited an infinite family of  such relations.  
  The smallest one  corresponds to the Ramanujan cusp form  of weight 12:
  \begin{equation}\label{modintro}  28 \, \zeta (3,9)+  150 \, \zeta(5,7) + 168  \, \zeta(7,5) = \frac{5197}{691} \zeta(12) \ .
\end{equation} The coefficients in this and  all such  equations can  be related to period polynomials for cusp forms, or equivalently, to group cocycles for $\mathrm{SL}_2(\Z)$. Furthermore,  a geometric mechanism for these relations is by now 
fairly well understood \cite{MMV}. 

The situation in higher depths  remains very unclear. It is known by work of Zagier and Goncharov that \eqref{introBK} is true (in a suitable setting, i.e.,  for solutions to the double shuffle equations as discussed below) in  depths 2 and 3.  Nevertheless, the presence of the term in $t^4$ in the right-hand side of \eqref{introBK} suggests  a new phenomenon in depth four. If we interpret the right-hand side of \eqref{introBK} in terms of the Poincar\'e series of a  depth-graded version of the Lie algebra of the category of mixed Tate motives over $\Z$, the  term in $t^4$ suggests the existence of 
 \emph{new generators} in this Lie algebra in depth four, corresponding to cusp forms for the full modular group.

In this paper we  supply  candidates for these  `exceptional' generators by constructing them explicitly out of period polynomials of cusp forms. As a result,  we  can formulate a much more precise conjecture than \eqref{introBK} which predicts not only the dimensions of the spaces of multiple zeta values in all weights and depths, but also their   relations    (modulo $\zeta(2)$ and modulo  lower depths).  In order to get some sense of these exceptional generators,  consider the first one, which occurs   in  depth 4 and weight 12.
 It turns out in this case that  all multiple zeta values are proportional to  a single element $\zeta(1,1,2,8)$ modulo terms of lower depth and products, e.g., 
$$\zeta(4,3,3,2) \equiv 116\,  \zeta(1,1,2,8)\ .$$ 
The exceptional generator   corresponding to  the Ramanujan cusp form  $\Delta$ annihilates every such relation, and therefore gives an interpretation of the coefficients (in this case, the number 116) in terms of ratios of critical values of  the $L$-function of $\Delta$.

\subsection{The projective line minus three points} %
 The  motivic Galois group  $$\mathcal{G}^{dR}_{\MT(\Z)}=\mathrm{Aut}^{\otimes}_{\MT(\Z)}(\omega_{dR})$$ of the Tannakian category of mixed Tate motives over $\Z$ has a canonical representation  
 $$\mathcal{G}^{dR}_{\MT(\Z)} \To \mathrm{Aut} \left( \opi\right) $$
 where $ \opi=\pi^{\mathrm{dR}}_1(\Pro^1 \backslash \{0,1,\infty\}, \tone_0, -\tone_1)$ is the de Rham fundamental torsor of paths  
between  tangential basepoints at $0$ and $1$ \cite{D,DG}.     This action is the motivic version of the  outer action of the absolute Galois group $\mathrm{Gal}(\overline{\Q}/\Q)$ on the pro-$\ell$ completion of the fundamental group of $X$ which was first studied extensively by Deligne, Drinfeld, Ihara \cite{D,Drinfeld,YI}.
Deligne conjectured that this action is faithful, or equivalently, that the motivic  torsor of paths  on $X$ generates the category $\MT(\Z)$.  
This was proven in \cite{BrMTZ}. Since the  Lie algebra of $\opi$ is the free  graded Lie algebra on two generators $e_0,e_1$, 
 $$\mathfrak{g} = \mathbb{L}( e_0,e_1) $$
  this gives a very concrete way to study the motivic Galois group, or equivalently, its graded Lie algebra $\gmzv$  together with its representation: 
  $$\gmzv:= \mathrm{Lie}\, \mathcal{G}^{dR}_{\MT(\Z)} \To  \mathrm{Der} \, \mathfrak{g}\ .$$
Since the derivations in question are `special', and can  be identified with elements of $\mathfrak{g}$, one obtains from this   a canonical embedding
 $$\gmzv \quad  \subset \quad  \left( \mathfrak{g}  , \{ , \}\right)  $$ 
 where  $\mathfrak{g}$ is equipped  with a new Lie algebra structure, denoted by $\{\ , \  \}$ and called the Ihara bracket. 
 By abuse of notation, we shall identify $\gmzv$ with its image in $(\mathfrak{g}, \{ \ , \})$.   A major problem is to describe this Lie algebra as precisely as possible.

We know that it has  generators  called `zeta elements' 
 for every $n\geq 1$:
$$\sigma_{2n+1} = (-1)^n (\mathrm{ad} \, e_0)^{2n} \, e_1 +\hbox{terms of degree } \geq 2 \hbox{ in  } e_1$$
but they are not canonical\footnote{in fact, one can define canonical generators $\sigma_{2n+1}$ to be the coefficient of 
$\zetam(3,2,\ldots, 2)$, with $n-1$ twos, in a motivic Drinfeld associator $\Phim= \sum_{w} \zetam(w) w$. However, this definition is not explicit and most of the  coefficients of $\sigma_{2n+1}$ defined in this way are not known explicitly.}
 (for example, $\sigma_{11}$  is only well-defined up to addition of rational multiples of $\{\sigma_3, \{\sigma_5, \sigma_3\}\}$).  Deligne's conjecture in its originally stated form was that  the $\ell$-adic versions of the $\sigma_{2n+1}$, which are elements in a certain Lie algebra constructed out of $\mathrm{Gal}(\overline{\Q}/\Q)$,  form a free Lie algebra for the Ihara bracket.

\begin{thm}\label{DI} \cite{BrMTZ} The graded Lie algebra $\gmzv$ is a   free Lie algebra on (some choice of) generators $\sigma_{2n+1}$ in each degree $-(2n+1)$ for $n \geq 1$.
\end{thm} 
Only the classes    $[\sigma_{2n+1}]$ in the abelianisation $(\gmzv)^{\mathrm{ab}}= \gmzv/\{\gmzv,\gmzv\}$ are canonical. Since $H_1(\gmzv;\Q)= (\gmzv)^{\mathrm{ab}}$, one can rephrase the previous theorem by saying that
\begin{eqnarray}  \label{freecohom}
H_1(\gmzv; \Q ) & \cong  & \bigoplus_{n\geq 1} [\sigma_{2n+1}]  \Q \\
H_i(\gmzv; \Q)  & = & 0  \quad \hbox{ for } \quad  i \geq 2  \ .\nonumber 
\end{eqnarray}
The abstract structure of the Lie algebra $\gmzv$ is therefore very simple, but all the subtle information about mixed Tate motives  over $\Z$ is encoded in the coefficients of the generators $\sigma_{2n+1}$, which  are not known explicitly. In this paper we study the depth-graded version of this Lie algebra. We gain an explicit description of all the generators, but the structure of the depth graded Lie algebra, which is more complicated, is conjectural.

\subsection{Depth filtration on the motivic Lie algebra} As explained in \cite{DG}, 
the depth filtration is induced geometrically by the inclusion 
$ X \subset \G_m $ and is the decreasing filtration $\dd$  on $\mathfrak{g}$,  where $\dd^r$ consists of Lie brackets containing at least $r$ occurences of the letter $e_1$. It is preserved by the Ihara bracket. Thus we obtain a depth filtration 
$\dd^r \gmzv$ on the motivic Lie algebra, and can consider its associated  graded Lie algebra:
$$\gd = \gr_{\dd} \gmzv\ .$$ 
It is bigraded by weight and depth: the depth will be indicated by a subscript, thus $\gd_r = \gr^r_{\dd} \gmzv$.
In depth one, it inherits canonical `zeta' generators  
\begin{equation} \label{introsigmabar}
\overline{\sigma}_{2n+1} = (-1)^n   (\mathrm{ad} \, e_0)^{2n}  e_1 \quad   \in \quad  \gd_1\ .
\end{equation}
Ihara discovered, astonishingly, that in the depth-graded Lie algebra the generators $\overline{\sigma}_{2n+1}$ are not free. The first relation occurs in weight 12
\begin{equation} \label{IntroIharaRel} \{\overline{\sigma}_3, \overline{\sigma}_9\} - 3 \{\overline{\sigma}_5, \overline{\sigma}_7\}=0\ .
\end{equation} 
In order to reconcile this relation with the freeness theorem \ref{DI}, there must exist an extra generator in $\gd$ in weight 12 to compensate for it.  These exceptional generators are one of the main objects of study of this paper.

The  general quadratic relations between the $\overline{\sigma}_{2n+1}$ have been known explicitly for some time  \cite{IT, Sch, GG, GKZ}   and are re-derived  very easily using the formalism we introduce  below.  To describe them,  let $V_{n} = \bigoplus_{i+j=n} \Q X^i Y^j$ denote the vector space of homogeneous polynomials of degree $n$ in two variables, with its right action of $\mathrm{SL}_2(\Z)$. Evaluating cocycles  on the matrix $\left(\begin{smallmatrix} 0 & -1 \\ 1 & 0 \end{smallmatrix}\right)$ induces  an isomorphism 
$$ H^1_{\mathrm{cusp}} ( \SL_2(\Z), V_{2n-2})^+ \quad  \cong  \quad  \Ss_{2n} $$
where $+$ denotes invariants under real Frobenius, and 
 $\Ss_{2n} \subset \Q[X,Y]$  is the space of even period polynomials:  it is the space  of antisymmetric homogeneous polynomials $P(X,Y)$  of degree  $2n-2$,   divisible by $Y$,  satisfying $P(\pm X, \pm Y)=P(X,Y)$ and  
$$P(X, Y) + P(X-Y,X) + P (-Y, X-Y) =0 \ . $$  
One shows that the quadratic relations are completely described by period polynomials: 
 \begin{equation} \label{introsigmarel}
\sum_{i<j}  \lambda_{i,j}  \{ \overline{\sigma}_{2i+1},  \overline{\sigma}_{2j+1}\}   = 0  \qquad \Longleftrightarrow \qquad     \sum_{i,j} ( \lambda_{ij}-\lambda_{ji})  X^{2i} Y^{2j} \in \Ss_{2n}\ .
\end{equation}
The first goal of this paper is to provide a complete conjectural description  of $\gd$.

\subsection{Results}
\subsubsection{Missing generators}
Firstly, we construct the candidate exceptional generators in depth 4  using period polynomials for cusp forms. We define an explicit map 
$$\e : H_{\mathrm{cusp}}^1(\SL_2(\Z); V_{2n-2})^+  \To   \dd^4\g$$
which, to every  even period polynomial  associates a Lie word in two generators $e_0,e_1$, of degree $4$ in $e_1$. 
The simplest possible conjecture that one can make is that the depth-graded motivic Lie algebra is generated by the canonical generators $\overline{\sigma}_{2n+1}$ in depth 1, the image of the exceptional map $\e$ in depth 4 and subject only to the known quadratic relations between the $\overline{\sigma}_{2n+1}$ in depth 2.
This statement is equivalent to   the following reformulation of the Broadhurst-Kreimer conjecture.

\begin{conj} \label{introconjMBK}  The image of $\e$ lies in $\gd$, and \begin{eqnarray}
H_1(\gd;\Q) &\cong&  \bigoplus_{n\geq 1} \overline{\sigma}_{2n+1} \Q \oplus  \bigoplus_{n} \e(\Ss_{2n})    \label{introfinalBK} \\
H_2(\gd;\Q) &\cong&  \bigoplus_n \Ss_{2n} \nonumber \\
H_i(\gd;\Q) &= &  0  \quad \hbox{ for } \quad  i\geq 3 \ .  \nonumber 
\end{eqnarray}
\end{conj}

 We show in \S \ref{sectEnum}  that  conjecture $\ref{introconjMBK}$   implies the version of $(\ref{introBK})$ in which multiple zeta values are replaced by motivic multiple zeta values. This is in turn equivalent to  $(\ref{introBK})$ if one assumes the period conjecture for mixed Tate motives over $\Z$.

 Conjecture \ref{introconjMBK}  describes    all relations between depth-graded motivic multiple zeta values (modulo products and  $\zetam(2)$): a linear relation of weight $N$ and depth $d$ is true if and only if 
 it holds amongst the coefficients of all Lie brackets of the elements $\overline{\sigma}_{2n+1}$ and $\e_P$  of the same total weight and depth (see  \ref{example12D4}).

 \subsubsection{Linearised double shuffle equations}
 The main evidence for the previous conjecture comes from the double shuffle equations.  It is known that
 $$\gmzv \subset \mathfrak{grt}  \subset \mathfrak{dmr} $$
 where $\mathfrak{grt}$ is Drinfeld's  Grothendieck-Teichm\"uller Lie algebra, and $\mathfrak{dmr}$ is Racinet's regularised double shuffle Lie algebra, both of which are defined by explicit equations. The inclusion  of $\gmzv$ in $\mathfrak{grt}$ and $\mathfrak{dmr}$ follow from theorem \ref{DI} and results of Drinfeld  \cite{Drinfeld} and Racinet  \cite{R} respectively. The inclusion $\mathfrak{grt} \subset \mathfrak{dm}_0$ is due to Furusho \cite{F}.
 
 If we pass to the depth-graded Lie algebras, we have
 $$\gd \subset  \gr^{\dd} \mathfrak{dmr} \subset \gdsh$$ 
 where $\gdsh$ are the linearized double shuffle equations defined in \cite{IKZ}. The advantage of these equations are  that they  are extremely simple to define: $\gdsh$ is essentially the intersection of two shuffle algebras.  Hitherto, $\gdsh$  was studied merely as a vector space, but it turns out that, as a consequence of the work of Racinet, that it also inherits a Lie algebra structure for the linearised Ihara bracket.    We offer a complete conjectural description of  $\gdsh$ below. 
  
 The first theorem states that the exceptional elements are solutions to the linearised double shuffle equations in depth 4. 
 \begin{thm}
 There is an explicit injective linear map
\begin{equation} \e:\Ss_{2n} \To \gdsh_4\ .
\end{equation}
\end{thm}
The formula for the map  $\e$  is given in \S \ref{sectExcep}, and associates   to every   even period polynomial $f$ a  solution $\e_f$ of the linearized double shuffle relations.   The question of whether the elements $\e_f$ 
are motivic, i.e., whether they lie in $\gd$, is open.

\begin{conj} \label{introconjMBKStrong}  One has  \begin{eqnarray}
H_1(\gdsh;\Q) &\cong&  \bigoplus_{n\geq 1} \overline{\sigma}_{2n+1} \Q \oplus  \bigoplus_{n} \e(\Ss_{2n})    \label{introfinalBK} \\
H_2(\gdsh;\Q) &\cong&  \bigoplus_n \Ss_{2n} \nonumber \\
H_i(\gdsh;\Q) &= &  0  \quad \hbox{ for } \quad  i\geq 3 \ .  \nonumber 
\end{eqnarray}
\end{conj}

This conjecture states that $\gdsh$ is generated by zeta elements and the exceptional generators in depth 4 subject only to the known quadratic relations between zeta elements. It is at the same time, the simplest, and  the strongest conjecture  that one can formulate. It implies several open conjectures about relations between multiple zeta values. For example, it implies: conjecture \ref{introconjMBK}, and hence the motivic version of the Broadhurst-Kreimer conjecture; the conjecture   $\gd=\gdsh$ (which in turn implies a conjecture in \cite{IKZ}); and the conjectures   $\gmzv= \mathfrak{grt}$ (Drinfeld)   and $\gmzv = \mathfrak{dm}_0$  (Zagier, Racinet).  The proofs of these implications use theorem \ref{DI} and Furusho's theorem \cite{F} in an essential  way. 
Since the Lie algebra $\gdsh$ is defined in a very simple and completely elementary way, the previous conjecture suggest a possibility of 
seeking a proof of all  of the above conjectures intrinsically within  the theory of  modular forms.

\subsubsection{Discussion}  
Note  that  the vanishing of $H_i(\gdsh)$ for $i\geq 3$ is equivalent to the vanishing for $i=3$. This follows from the well-known fact that the vanishing of a Yoneda Ext group causes all higher Ext groups to vanish (\cite{EL},  \cite{BrZeta} remark 8.6). 

The exceptional elements $\e$ satisfy a range of special  properties which are studied in \S\ref{sectExcep}, \ref{sectEfProp}. Using these, we can prove that they satisfy no quadratic relations, which provides some meagre evidence in favour of the previous conjectures. 
We  do not know, however, that the $\e_f$ are non-trivial in the abelianisation of $\gdsh$, nor can we presently rule out the existence of relations of the form
$$\{\e_f, \overline{\sigma}_{2n+1} \}  \  \in \   \mathrm{Lie}_5  \, \gdsh_1$$
which can only occur in  depth $\geq 5$ and weight $\geq 15$.   Relations which are quadratic in  the $\e_f$ could first occur in weight 28 and depth 8. Viewed  from this perspective, the current numerical data in favour of the standard conjectures on multiple zeta values is  lacking, since new phenomena could potentially occur in weights and depths beyond the range of current experimentation.   Any such phenomena  would point to new and fascinating connections  between mixed Tate motives over $\Z$ with geometry and arithmetic.  Indeed,  the methods introduced in this paper should  enable one to test the validity of conjecture \ref{introconjMBKStrong} to far higher weights than presently known.

The motivic Lie algebra $\gmzv$ together with its depth filtration $\dd$ naturally gives rise to a spectral sequence \S\ref{sect:DepthSpectralSequence}, and the Broadhurst-Kreimer conjecture is equivalent to the statement that this spectral sequence should behave as simply as possible (given the existence of the quadratic relations between $\overline{\sigma}_{2n+1}$); there is only one non-trivial differential: 
$$d  : \Ss_{2n} \To (\gd_4)^{\mathrm{ab}}$$
This is possibly also an argument in  favour of conjecture \ref{introconjMBK}.  In \cite{BrZeta} we computed this differential by finding canonical lifts of the zeta elements to depth 3. The relation with the map $\e$ is mysterious, although Yasuda has subsequently found a conjectural relation between and $d$ and the image of $\e$ in the abelianisation of $\gdsh$ (private correspondence) which  involving critical values of $L$-functions. See the computations in example \ref{exEmot}.

Finally, we investigate  the Lie subalgebra of $\gd$ which is generated only by the elements $\overline{\sigma}_{2n+1}$ (without exceptional elements), and conjecture that it  describes the structure of \emph{totally odd} depth-graded motivic multiple zeta values.

\subsection{Contents of the paper}
In sections \S\S \ref{sectreminders}-\ref{sectdepth} we recall some background on the motivic fundamental group of $\Pro^1\backslash \{ 0, 1, \infty\}$, the Ihara action and  the depth filtration. In \S \ref{sectLinrel}, we discuss the linearized double
shuffle relations from the Hopf algebra point of view.
In \S \ref{sectPolyrep}, and throughout the rest of the paper, we use polynomial representations  to replace words of fixed $\dd$-degree $r$ in $e_0,e_1$ with polynomials in $r$ variables:
$$
\overline{\rho} \ : \  \gr_r^{\dd}  \mathfrak{g}  \To  \Q[x_1,\ldots, x_r]  $$
which sends words beginning in $e_0$ to zero, and   $\overline{\rho} (e_1 e_0^{n_1}  \ldots e_1 e_0^{n_r})=x_1^{n_1}\ldots x_r^{n_r}$.
This replaces identities between non-commutative formal power series with functional equations in commutative polynomials, strongly reminiscent of those considered in  \cite{Ecalle, IKZ}.  
We show that in the polynomial representation, the Ihara bracket
has an extremely simple form   (\S \ref{sectdihedihara}). A simple way to view the duality between depth-graded motivic multiple zeta values and polynomials is via the generating series 
$$\sum_{n_1,\ldots, n_r \geq 1}  \zetam_{\dd}(n_1,\ldots, n_r)\,  x_1^{n_1-1} \ldots x_r^{n_r-1}\ .$$
 The  generators $(\ref{introsigmabar})$ are simply  the coefficients of $\zetam_{\dd}(2n+1)$: 
$$\overline{\rho} (\overline{\sigma}_{2n+1}) = x_1^{2n} \quad \hbox{ for } n\geq 1 \ .$$
    In \S \ref{sectmod2} we review the  relation between period polynomials and depth two multiple zeta values, and in \S\ref{sectExcep}, we define for each period polynomial 
   $P$, an element
   $$\overline{\rho}(e_P)  \in \Q[x_1,x_2, x_3, x_4]\ ,$$
   which defines the exceptional elements in $\gdsh_4$. These elements satisfy some remarkable properties (\S \ref{sectEfProp})  which are stable under the Ihara bracket. 
In \S\ref{sectinterp} we discuss conjecture \ref{introconjMBK} and its consequences, and in \S\ref{secttotallyodd} we discuss some applications for the enumeration of the totally odd multiple zeta values
$\zeta(2n_1+1, \ldots, 2n_r+1)$ where $n_i\geq 1$.

\subsubsection{Related work}
Since the first draft of this paper appeared, there have been a number of related developments, which cannot all be mentioned here for reasons of space. We apologise to the many people, who have subsequently extended some of the ideas in this paper in various directions \cite{Matthes,  Tasaka, Ma, Li}, not only for the delay in publishing this work, but also for being unable  to give a complete survey of subsequent developments of the subject here.

First of all, an alternative description of depth 4 exceptional generators was given in \cite{BrZeta}, by constructing canonical zeta elements in $\gmzv/\dd^4 \gmzv$ and applying the differential in the depth spectral sequence.  These elements  are motivic (i.e., lie in $\gd_4$), but 
their relationship to the exceptional elements $\e$ defined here is far from  clear.   The entire construction  relies on the relationship between $\Pro^1 \backslash \{0,1,\infty\}$  (genus 0) and  the unipotent completion of the fundamental group Tate elliptic curve (genus 1), and the fact that the depth filtration is induced by natural filtrations in the elliptic setting.  Similarly, the precise relationship between the work of Pollack \cite{Pollack} is now  partially understood but warrants further investigation. 

Subsequently, the work \cite{MMV}  explains the origin of the quadratic relations between $\overline{\sigma}_{2n+1}$  by proving that they come from `modular elements'  corresponding to non-critical values of L-functions of cusp forms, which act on the relative completion of the fundamental group on the moduli stack $\mathcal{M}_{1,1}$ of elliptic curves. From this perspective, the canonical generators $\overline{\sigma}_{2n+1}$ can be understood as coming from Eisenstein series, which this supports the philosophy that conjecture \ref{introconjMBKStrong} relates entirely to  modular forms. Nevertheless,  the connection between modular elements and the phenomena in depth 4  studied in the present paper remain unexplored.

In a very different direction, Goncharov  \cite{GG}  has studied  the depth filtration from the  perspective of the  homology of the general  linear group $\mathrm{GL}_d(\Z)$, which he relatest to  the structure of multiple zeta values in depth $d$. We have no understanding of the relationship between this point of view and  conjecture $(\ref{introfinalBK})$.

After the first draft of this paper appeared, the conjecture \ref{conjmain} was related to questions of Koszulity in the paper  \cite{EL}.

Finally, the interested reader may wish to consult the notes\footnote{\url{www.ihes.fr/~brown/BKExactSeq1.pdf}}, where, in   answer to a question due to Zagier, we  reformulated the Broadhurst-Kreimer conjecture as a short exact sequence, which suggests another line of attack on the conjecture \ref{introconjMBKStrong}. 
\\

\emph{Acknowledgements}.  
 This work was partially supported by ERC grants PAGAP, ref. 257638 and GALOP, ref. 724638. I would like to thank the Institute for Advanced Study and the  Humboldt Foundation. This work was partly inspired by  Zagier's lectures at the Coll\`ege de France in October 2011, in which he explained parts of 
\cite{IKZ}, \S8, and partly based on a talk at a conference  in June 2012 at the IHES in honour of P. Cartier's eightieth birthday. It is a pleasure to dedicate this work to him, with my deepest gratitude.

\section{Reminders on $\pi_1^{\mathfrak{m}}(\Pro^1\backslash \{0,1,\infty\})$} \label{sectreminders}

\subsection{The motivic $\pi_1$ of $ \Pro^1\backslash \{0,1,\infty\}$} \label{sect21} Let $X=\Pro^1\backslash \{0,1,\infty\}$, and let 
$\tone_0, -\tone_1$
denote the  tangential base points on $X$  given by the tangent vector  $1$ at $0$, and  the tangent vector $-1$ at $1$. Denote the de Rham realization of the motivic fundamental torsor of paths on $X$ with respect to these basepoints  by: 
$$\opi= \pi^{dR}_1 ( X,\tone_0, -\tone_1)  \ .$$
It is the affine scheme over $\Q$ which to any commutative unitary $\Q$-algebra $R$ associates the set of group-like formal power series in two non-commuting variables $e_0$ and $e_1$
$$\{ S \in R\langle \!\langle e_0, e_1 \rangle \! \rangle^\times:  \Delta S= S \widehat{\otimes} S\}\ ,$$
where $\Delta$ is the completed coproduct for which the elements $e_0$ and $e_1$ are primitive: $\Delta e_i = 1\otimes e_i + e_i \otimes 1$ for $i=0,1$. The ring of regular functions on $\opi$ is the $\Q$-algebra
$$\Or(\opi)\cong \Q \langle e^0, e^1 \rangle$$
whose underlying vector space is spanned by the set of words $w$ in the letters $e^0, e^1$, together with the empty word, and  whose multiplication is given by
 the shuffle product $\sha:  \Q \langle e^0, e^1 \rangle\otimes_{\Q}  \Q \langle e^0, e^1 \rangle \rightarrow  \Q \langle e^0, e^1 \rangle $.
 The deconcatenation of words defines a coproduct, making $\Or(\opi)$ into a Hopf algebra. This gives rise to a group structure on  $\opi(\Q)$ (corresponding to the fact that in the de Rham realisation there is a canonical path between any two points on $X$).   Any word $w$ in $e^0,e^1$  defines a function 
$$\opi(R) \To R$$
which extracts the coefficient $S_w$ of the word $w$ (viewed in $e_0,e_1$)  in  a group-like series $S \in R\langle \!\langle e_0, e_1 \rangle \! \rangle^{\times} $. 
The Lie algebra of  $\opi(\Q)$  is the completed Lie algebra   $\g^\wedge$ of the graded Lie algebra $\g$  which is freely generated by the two elements  $e_0,e_1$ in degree  minus one. The universal enveloping algebra of $\g$  is the tensor algebra on $e_0,e_1$:
\begin{equation} \label{Ugdescription}
\U \g = \bigoplus_{n\geq 0}  (\Q\, e_0 \oplus \Q \, e_1)^{\otimes n} \ .
\end{equation} 
It is the graded cocommutative Hopf algebra which is the graded dual of $\Or(\opi)$. Its multiplication is given by  the concatenation product, and its coproduct is the   unique coproduct for which  $e_0$ and  $e_1$ are primitive.

\subsection{Action of the motivic Galois group} Now let $\MT(\Z)$ denote the Tannakian category of mixed Tate motives over $\Z$, with canonical  fiber functor   given by the de Rham realization.  Let $\GMT$ denote the group of automorphisms of this fiber functor.  It is an affine group scheme over $\Q$. It has a decomposition as a semi-direct product
$$ \GMT\cong  \UMT \rtimes \mathbb{G}_m \ ,$$
where $\UMT$ is pro-unipotent. Furthermore, one knows from the relationship between the Ext groups in $\MT(\Z)$ and Borel's results on  the rational algebraic $K$-theory of $\Q$ that 
the graded Lie algebra of $\UMT$ is non-canonically isomorphic to the  Lie algebra freely generated by  one generator $\sigma_{2i+1}$ in degree  $-(2i+1)$ for every $i\geq 1$.
It is important to note that only the  classes of the elements $\sigma_{2i+1} $ in the abelianization $\UMT^{ab}$ are canonical,  the elements $\sigma_{2i+1}$ themselves are not.
    
    Since  $\Or(\opi)$ is the de Rham realization of an  Ind-object in the category $\MT(\Z)$, there is an action of the motivic Galois group
    \begin{equation} \label{UMTaction} \UMT \times \opi \To \opi\ . \end{equation}
   The action of $\UMT$ on the unit element $1\in \opi$  defines a map
   \begin{equation} \label{gtog1} g\mapsto g(1): \UMT \To \opi\ ,
   \end{equation}
   and the action $(\ref{UMTaction})$ factors through a map   
\begin{equation}\label{action2}
       \circ : \opi \times \opi  \To  \opi  
\end{equation}
 first computed by Y. Ihara. 
It is  obtained from \cite{DG}, \S5.9, \S5.13 by reversing all words in order to be consistent with our conventions for iterated integrals.
 Given  $a\in \opi$, its action  on an element $g\in \opi$ is
 \begin{equation}Ê\label{IharaFullAction}    a\circ g =   (\langle a\rangle_0 (g)) \, a\ ,
 \end{equation}
 where $\langle a \rangle_0$ acts on the generators 
 $\exp(e_i)$, for $i=0,1$, by
\begin{eqnarray} \label{a0action}
\langle a \rangle_0 ( \exp(e_0))  & = &  \exp(e_0)  \\
 \langle a \rangle_0 (\exp(e_1)) & = & a \exp(e_1) a^{-1}  \ .\nonumber 
\end{eqnarray}

\begin{defn} Define a $\Q$-bilinear map 
$$ \circb  : \U \g \otimes_{\Q} \U \g \rightarrow \U \g \ $$
inductively as follows. For   any words $a, w$ in $e_0,e_1$, and for any integer  $n\geq 0$,   let 
\begin{equation}\label{circbdef}
 a \circb  (e_0^{n} e_1 w) =        e_0^n  a e_1  w  +   e_0^n e_1 a^* w +e_0^n e_1 (a \circb w)   \end{equation}
where  $a \circb e_0^n =  e_0^n \, a $, and  for any $a_i \in \{e_0,e_1\}$, 
 $(a_1\ldots a_n)^* = (-1)^n a_n \ldots a_1.$
 \end{defn} 
The map $\circb$ is  a linearisation of the full action $\circ : \U \g \otimes_{\Q} \U \g \rightarrow \U \g$ induced by $(\ref{UMTaction})$.

\begin{prop} \label{propcircscompared} Let $\circ$ also  denote the action 
\begin{equation} \label{gUgtoUg}Ê \g \otimes_{\Q} \U \g \rightarrow \U \g \end{equation}
induced by $(\ref{action2})$. If $i:\g \rightarrow \U \g$ is the natural map, then 
$a \circ b = i(a) \circb b$.
\end{prop}
\begin{proof} For any $S \in \opi$, the coefficient of $w$ in $S^{-1}$ is  equal to the coefficient of $w^*$ in $S$, since the map $*$ is the antipode in $\Or(\opi)$.  Identifying $\g$ with its image in $\U \g$, it follows that  the infinitesimal, weight-graded,  version 
of  the map $(\ref{a0action})$  is the derivation $\langle f\rangle_0: \g \rightarrow \g$ which  for any $f\in \g$ is given by  $$e_0\mapsto 0 \quad ,  \quad e_1 \mapsto f e_1 + e_1 f^*\ . $$ Thus  $\langle f\rangle_0 : \U \g \rightarrow \U \g $ is the map
\begin{multline} \nonumber \langle f \rangle_0  \, (e_0^{m_0} e_1 \ldots e_1 e_0^{m_r} ) = e_0^{m_0}(f  e_1 +e_1 f^*) e_0^{m_1} e_1 \cdots e_0^{m_{r-1}} e_1 e_0^{m_r} +    \\
 \ldots + e_0^{m_0} e_1 e_0^{m_1} \cdots  e_0^{m_{r-1}}(f e_1+ e_1 f^*) e_0^{m_r} \ .  
\end{multline}
 Concatentating on the right by $f$ gives precisely the map defined by $(\ref{circbdef})$.
\end{proof}

\subsection{The motivic Lie algebra} By  equation $(\ref{gtog1})$, we obtain a map of Lie algebras
\begin{equation}  \label{MLA}
\Lie^{gr} (\UMT) \To  \left(\g \ , \ \{ , \} \right)  \ ,
\end{equation}
where the Ihara bracket satisfies $\{f,g\}= f\circ g- f \circ f$. 
It   follows from theorem \ref{DI} that this map is injective \cite{BrMTZ}. In this paper, we shall identify 
$\Lie^{gr}(\UMT)$ with its image, and  abusively call it  the motivic Lie algebra.

\begin{defn} The  motivic Lie algebra $\gmzv \subseteq \g$  is  the image of the map $(\ref{MLA})$.
\end{defn} 

The Lie algebra $\gmzv$ is non-canonically isomorphic to the free Lie algebra with one generator $\sigma_{2i+1}$ in each degree $-(2i+1)$ for $i \geq 1$.

\section{Motivic multiple zeta values}
Let $\Amt$ denote the graded  Hopf algebra of  functions on $\UMT$. Dualizing $(\ref{UMTaction})$ gives the motivic coaction (written in this paper as a left coaction):
$$\Delta^{\Mot}: \Or(\opi) \To \Amt \otimes_{\Q} \Or(\opi)\ .$$
Furthermore, the de Rham image of the straight path $\dch$ from $0$ to $1$ in $X$ defines the Drinfeld associator element $\Phi\in \opi(\R)$ which begins
$$ \Phi= 1 + \zeta(2) [e_1,e_0] +  \zeta(3) ([e_1,[e_1,e_0]]+ [e_0,[e_0,e_1]] ) + \ldots  $$
The map which takes the coefficient of a word $w$ in $\Phi$ defines the period homomorphism
$$per : \Or(\opi) \To \R\ .$$
Here, we use the convention from \cite{BrMTZ}: the  coefficient of the word $e_{a_1}\ldots e_{a_n}$ in $\Phi$, for $a_i \in \{0,1\}$, is the  iterated integral 
$$\int_0^1  \omega_{a_1} \ldots \omega_{a_n} $$
regularized with respect to the tangent vector $1$ at $0$, and $-1$ at $1$, where the integration begins on  the left, and 
$\omega_0 = {dt \over t}$ and $\omega_1 = {dt \over 1-t}$. 
\begin{defn}  \label{def3.1} The algebra of motivic multiple zeta values is defined as follows. The ideal of motivic relations between MZV's is defined to be $\Jmt\leq \Or(\opi)$,  the largest graded ideal  contained in the kernel of $per$ which 
is stable under $\Delta^{\Mot}$.   We set
$$\Ho = \Or(\opi)/ \Jmt\ , $$
and let $\zetam(n_1,\ldots, n_r)$ denote the image of the word $e^1(e^0)^{n_1-1} \ldots e^1 (e^0)^{n_r-1}$ in $\Ho$. Likewise,  for any $a_1,\ldots, a_n \in \{0,1\}$,  we let
$\Imot(0;a_1,\ldots, a_n;1)$ denote the image of the word $e^{a_1}\ldots e^{a_n}$ in $\Ho$.
Equivalently, one can define   $$\Imot(0;a_1,\ldots, a_n;1) = [\Or(\pi_1^{mot}(X, \tone_1, -\tone_0), \mathrm{dch}, e^{a_1}\ldots e^{a_n} ]^{\mathfrak{m}}$$ 
 where the right-hand side is a motivic period \cite{BrICM} of the category $\MT(\Z)$, and define the motivic multiple zeta values by specialising to the case when $a_1=1$ and $a_n=0$. 
\end{defn}

Then $\Ho$ is naturally graded by the weight, has a graded coaction
\begin{equation}\label{deltacoaction}
\Delta^{\Mot}: \Ho \To \Amt \otimes_{\Q} \Ho \end{equation}
and a period map $per:  \Ho \rightarrow \R$. The period of $\zetam(n_1,\ldots, n_r)$ is $\zeta(n_1,\ldots, n_r)$.
One obtains partial information about the  motivic coaction $(\ref{deltacoaction})$ using the fact that it factors through the 
coaction which is dual to the Ihara action $(\ref{action2})$.

\subsection{The Ihara coaction} \label{sectIharacoact} 
For any graded Hopf algebra $H$, let $IH= H_{>0}/H^2_{>0}$ denote the Lie coalgebra of indecomposable elements of $H$, and let $\pi: H_{>0} \rightarrow IH$ denote the natural map.
  Dualizing  $(\ref{gUgtoUg})$ gives an infinitesimal 
coaction
\begin{equation} \Or(\opi) \To I \Or(\opi) \otimes_{\Q} \Or(\opi) \ .
\end{equation}
Let $D_{r}:  \Or(\opi) \rightarrow I \Or(\opi)_r \otimes_{\Q} \Or(\opi)$ denote its  component of degree $(r,\cdot)$, and let us denote the element $e^{a_1}\ldots e^{a_n}$ in $\Or(\opi)$ 
by $\II(0;a_1,\ldots, a_n;1)$, where $a_i \in \{0,1\}$. 

\begin{prop} Set $a_0=0, a_{n+1}=1$.  For all $r\geq 1$, and 
$ a_1,\ldots, a_n \in \{0,1\}$,
 \begin{multline}\label{mainformula}   
 D_r  \, \II( 0; a_1, \ldots, a_n;  1 ) =   \\ 
\sum_{p=0}^{n-r} \pi( \II(a_{p} ;a_{p+1},\tdots, a_{p+r}; a_{p+r+1})) \otimes   \II(0; a_1, \tdots, a_{p}, a_{p+r+1}, \tdots ,  a_n ; 1) \ .
 \end{multline}
where $\II(a_{p} ;a_{p+1},\tdots, a_{p+r}; a_{p+r+1}) \in \Or(\opi)$ is defined to be zero if $a_p=a_{p+r+1}$, and equal to 
$(-1)^r \II( a_{p+r+1} ;a_{p+r},\tdots, a_{p+1}; a_{p})$
if $a_p=1$ and $a_{p+r+1}=0$.
\end{prop}
\begin{proof} One checks that this formula is dual to $\circb$ (see also \cite{BrICM}, \S2.5).   
\end{proof}
Since the motivic coaction  on $\Ho$ factors through the Ihara coaction,  it follows that the degree $(r,\cdot)$ component factors through operators
$$D_r:  \Ho \To I \Ao \otimes_{\Q} \Ho$$
given by the same formula as $(\ref{mainformula})$ in which each   term $\II$ is replaced by its image $\Imot$ in $\Ho$ (resp. $\Ao$).
Since $\Amt$ is cogenerated in odd degrees only, the motivic coaction on $\Ho$ is completely determined by  the set of operators $D_{2r+1}$ for all $r\geq 1$ (see \cite{BrMTZ}).

\section{The depth filtration} \label{sectdepth} 

\subsection{Definition} The inclusion 
$  \Pro^1\backslash \{0,1,\infty\} \hookrightarrow  \Pro^1\backslash \{0,\infty\}$ induces  a map on the motivic fundamental groupoids (\cite{DG}, \S6.1)
\begin{equation} \label{depthpimot} \pi^{mot}_1(X, \tone_0, -\tone_1) \rightarrow \pi^{mot}_1(\G_m, \tone_0, -\tone_1) \  ,
\end{equation}
and hence on the de Rham realisations
$$\Or(\pi^{dR}_1(\G_m,  \tone_0, -\tone_1))\To \Or(\opi)$$
which is  the inclusion of 
$ \Q\langle e^0 \rangle$ into $\Q \langle e^0, e^1 \rangle $. Define the image to be $\dd_0 \Or(\opi)$. It generates, via the deconcatenation coproduct, an increasing filtration called the depth:
\begin{equation} \label{depthasdege1}
\dd_i \Or(\opi) = \langle \hbox{words } w \hbox{ such that } \deg_{e^1} w \leq i \rangle_{\Q}\ , 
\end{equation}
with respect to which $\Or(\opi)$ is a filtered Hopf algebra. Furthermore, since the map $(\ref{depthpimot})$ is motivic, the depth filtration is 
preserved by the coaction  $(\ref{deltacoaction})$ (which also follows by direct computation), and  descends to the algebra $\Ho$.  By definition \ref{def3.1}  the depth filtration $\dd_d \Ho$ is the increasing filtration  defined by
$$\dd_d \Ho = \langle \zetam(n_1,\ldots, n_i): i\leq d\rangle_{\Q}\ .$$
Following \cite{DIHES}, it is convenient to define the $\dd$-degree on $\Or (\opi)$ to be  the degree in $e^1$.  It  defines a grading on $\opi$ which is not motivic.   By $(\ref{depthasdege1})$,  the depth filtration (which is motivic) is simply the increasing filtration associated to the $\dd$-degree. 

\subsection{Depth-graded motivic multiple zeta values}

\begin{defn}  Let  $\dd_d \Ao$ be the induced filtration  on the quotient $\Ao= \Ho/\zetam(2) \Ho$. 
 We define the \emph{depth-graded motivic multiple zeta value}    $$\zetam_{\dd} (n_1,\ldots, n_r)$$ to be the image of 
 $\zetam(n_1,\ldots, n_r)$ in  $\gr_r^{\dd} \Ho$. \end{defn}
 
 \begin{prop} There is a  non-canonical isomorphism of bigraded vector spaces:
 \begin{equation} \label{DHoisDAtensZeta2}
\gr^{\dd} \Ho  \cong \gr^{\dd}  \Ao  \otimes_{\Q}  \Q[\zetam_{\dd}(2)] \ ,
\end{equation} 
where $\zetam_{\dd}(2)^n$ are in depth 1 for  all $n\geq 1$. 
\end{prop}
\begin{proof}
Choose a homomorphism $\pi_2: \Ho \rightarrow \Q[\zetam(2)]$ which respects the weight-grading and such that $\pi_2(\zetam(2))=\zetam(2)$. Such a homomorphism
exists by  \cite{DG}, \S5.20 (see \cite{BrMTZ}, \S2.3). Composing with the coaction $(\ref{deltacoaction})$, we obtain a map
$$\Ho \overset{\Delta^{\Mot}}{\To} \Ao \otimes_{\Q}\Ho  \overset{id \otimes \pi_2}{\To} \Ao \otimes_{\Q}\Q[\zetam(2)]   \ . $$
By a motivic version  \cite{BrMTZ} of Euler's theorem, $\zetam(2)^n$ is a rational multiple of $\zetam(2n)$,  and so all powers of $\zetam(2)$ have depth one. The depth filtration thus satisfies 
$$\Q=\dd_0\Q[\zetam(2)] \  \subset\  \dd_1 \Q[\zetam(2)] = \Q[\zetam(2)] \ .$$
Since  $\dd_0 \Ho=\Q$,  it follows trivially that $\pi_2$ respects the depth filtration because  it is $\Q$-linear.
Since the depth filtration is motivic, it is also respected by  $\Delta^{\Mot}$, and therefore the composition
$ \Ho\rightarrow \Ao \otimes_{\Q}\Q[\zetam(2)]  $     also  respects the depth filtration.  The statement \eqref{DHoisDAtensZeta2}  follows since we  know that it is an isomorphism 
 (by (2.13) in \cite{BrMTZ} or   5.8 in \cite{DIHES}). 
\end{proof}

\subsection{Depth-graded motivic Lie algebra}

The depth filtration defines  a decreasing filtration $\dd^r$ on $\gmzv$  where $\dd^r$ consists of words with at  least $r$ occurrences of $e_1$: 
$$\dd^r  \gmzv = \langle w \in \gmzv: \deg_{e_1} w \geq r\rangle \ ,$$
We denote the associated graded Lie algebra 
 by $\gd$. There is correspondingly a decreasing  depth filtration on $\U \gmzv$, also denoted by $\dd$.

It follows from the definitions above that $\gd$ is the bigraded Lie algebra  dual to the bigraded coalgebra $\gr^{\dd} \, I  \mathcal{A}$ of depth-graded motivic multiple zeta values modulo products.  Thus the problem of studying relations between (motivic) multiple zeta values modulo lower depth (and modulo $\zetam(2)$) and the algebra $\gd$ are equivalent.

\subsection{Depth-parity}
The following proposition is a consequence of Tsumura's result  on double shuffle equations (proposition \ref{propparity}).

\begin{prop}  \label{propDepth-Parity} The depth-graded motivic Lie algebra $\gd$ vanishes  in bidegrees with different parity. 
More precisely, it vanishes in  weight $N$ and depth $r$ if $N\neq r \pmod 2$. 
\end{prop} 

Equivalently, if $n_1+\ldots+ n_r \not \equiv r  \pmod 2$, and $n_1+\ldots + n_r >2$ then 
\begin{equation} \label{vanishingzetamd}
\zetam_{\dd}(n_1,\ldots, n_r) \equiv 0 \pmod{\hbox{products}}\ .
\end{equation}
We do not need to work modulo $\zetam_{\dd}(2)$ in the previous equation since the weight is larger than $2$ by assumption and all other even zeta values $\zetam(2n)$, $n\geq 2$,  are products. 

\subsection{Depth spectral sequence} \label{sect:DepthSpectralSequence} The depth filtration on the motivic Lie algebra $\gmzv$ induces a homology spectral sequence which converges to associated graded for the  depth of the homology of $\gmzv$. By theorem \ref{DI}, the latter satisfies  
$$H_i(\gmzv) = \begin{cases} \bigoplus_{n\geq 1}  \Q [\sigma_{2n+1}]   \qquad \hbox{ if  } i =1  \\  0 \qquad \qquad  \qquad \qquad \hbox{ if } i \geq 1 \end{cases}\ ,  $$ 
and is entirely concentrated in depth $1$. 

The depth spectral sequence  satisfies 
$$E^1_{-p,q} = \gr_{\dd}^p H_{q-p} (\gd)$$
where $E^1_{-p,q}$ vanishes unless $p\geq 1$ and $p<  q \leq 2p$, as can easily be seen from the Chevalley-Eilenberg chain complex which computes the homology of a Lie algebra (see \S\ref{sectEnum} and \eqref{CE2}).
The differentials satisfy
$$d^r_{p,q} : E^r_{p,q} \rightarrow E^{r}_{p-r,q+r-1}\ .$$
\begin{prop} The differentials  $d^r$ vanish if $r$ is odd.
\end{prop} 
\begin{proof} The weight grading on $\gd$ induces a weight grading on the depth spectral sequence, with respect to which the differentials are of degree zero.
  It follows from proposition \ref{propDepth-Parity} and
  $E^1_{-p,q} = \gr_{\dd}^p H_{q-p} (\gd)$ that $E^1_{p,q}$, and hence all $E^r_{p,q}$,  vanish unless the depth and weight have the same parity. 
  The result follows since $d^r$ has (weight, depth)-bidegree $(0,r)$.
   \end{proof} 

Here follows  a picture of the potentially non-vanishing $E^1=E^2$ terms:
\begin{center}
\begin{tabular}{ccccc|c}
$\ddots$ & $\gr^4_{\dd}H_4$ & & &  &     8\\ 
 & $\gr^4_{\dd}H_3$ & & &  &     7\\ 
& $\gr^4_{\dd}H_2$ & $\gr^3_{\dd}H_3$ & &  &     6 \\ 
&$\gr^4_{\dd}H_1$ & $\gr^3_{\dd}H_2$ &  &  &     5 \\ 
&& $\gr^3_{\dd}H_1$ & $\gr^2_{\dd} H_2$ &  &     4 \\ 
&& & $\gr^2_{\dd} H_1$ &  &     3 \\ 
&& & & $\gr^1_{\dd}H_1$ &     2 \\ 
  &&  & &  &     1 \\ \hline
  \ldots & -4& -3 & -2 & -1   
\end{tabular}
\end{center}
In fact, we know by theorem \ref{DI} that:
 $$ \gr^1_{\dd}H_1 (\gd) \cong H_1(\gmzv) \cong  \bigoplus_{n\geq 1} \Q [\sigma_{2n+1}]\ ,$$ 
 and therefore everything to the left  of  the column indexed $-1$   converges to zero. 
Furthermore, one knows that  $\gr^2_{\dd} H_1 (\gd)=0$,  and $\gr^3_{\dd} H_i (\gd)$ vanish for all $i$. One also has  a complete description of $\gr^2_{\dd} H_2(\gd)$ in terms of period polynomials of cusp forms, as discussed below. Therefore,
the first  interesting part of the differential is
$$d^2: \gr^2_{\dd} H_2(\gd) \To  \gr^4_{\dd}H_1(\gd)\ ,$$
and the main conjecture \ref{introconjMBK} can be reformulated as saying that the components of all other differentials in the depth spectral sequence vanish.  
\section{Linearized double shuffle relations} \label{sectLinrel}

\subsection{Reminders on the standard relations}
We briefly review the double shuffle relations and their depth-linearized versions.  See \cite{Cartier, R, Anatomy} for further background.

\subsubsection{Shuffle product} Consider the algebra $\Q\langle e_0, e_1 \rangle$ of words in the two letters $e_0, e_1$, equipped with the shuffle product $\sha$. It is defined recursively by
\begin{equation} \label{shuffprod}
e_i w \sha e_j w' = e_i (w \sha e_j w') + e_j (e_i w \sha w')
\end{equation} 
for all $w,w'\in \{e_0,e_1\}^\times$ and $i,j\in \{0,1\}$, and the property that the empty word $1$ satisfies $1\sha w=w\sha 1=w$.  It is a Hopf algebra for the deconcatenation coproduct.  A linear map $\Phi: \Q\langle e_0, e_1 \rangle\rightarrow \Q$ is a homomorphism for the shuffle multiplication, 
or $\Phi_{w} \Phi_{w'} = \Phi_{w \sha w'}$ for all $w,w'\in \{e_0,e_1\}^\times$ and $\Phi_1=1$,  if and only if  the series
$$\Phi=\sum_w \Phi_w w  \in \Q \langle \langle e_0,e_1 \rangle \rangle$$
is invertible and group-like for the (completed) coproduct $\Deltas$ with respect to which $e_0$ and $e_1$ are primitive (compare \S\ref{sect21}). In other words, there is an equivalence:
\begin{equation} \label{shuffrel}
\Phi  \hbox{  homomorphism for  } \sha \quad \Longleftrightarrow  \quad \Phi   \in \Q \langle \langle e_0,e_1 \rangle \rangle^{\times} \hbox{ and } \Deltas \Phi= \Phi \widehat{\otimes} \Phi
\end{equation} 
One says that $\Phi$ satisfies the shuffle relations if either of the equivalent conditions $(\ref{shuffrel})$ holds. Passing  to the corresponding Lie algebra, we have an equivalence
\begin{equation} \label{shuffrelmodprod}
\Phi_{w\sha w'} = 0 \hbox{  for all } w,w' \quad \Longleftrightarrow  \quad \Phi   \in \Q \langle \langle e_0,e_1 \rangle \rangle \hbox{ and } \Deltas \Phi= 1 \widehat{\otimes} \Phi + \Phi \widehat{\otimes} 1
\end{equation} 
One says that $\Phi$ satisfies the shuffle relations \emph{modulo products} if either of the equivalent conditions $(\ref{shuffrelmodprod})$ holds.

The algebra $\Q\langle e_0, e_1\rangle$ is bigraded for the degree, or weight (for which  $e_0,e_1$ both have degree $-1$), and the $\dd$-degree 
for which  $e_1$ has degree $1$, and $e_0$ degree $0$. The relations $(\ref{shuffprod})$, $(\ref{shuffrel})$, $(\ref{shuffrelmodprod})$ evidently respect both gradings. 
In this case then, passing to the depth grading does not change the relations in any way, and the linearized shuffle relations are identical to the shuffle relations modulo products.

\begin{defn} Let $\Phi\in \gr_{\dd}^r \U \g$ be of weight $N$.  It defines a linear form $w\mapsto \Phi_w$ on words of  weight $N\geq 2$ and $\dd$-degree $r$. It
 satisfies the \emph{linearized shuffle relations} if 
  $$\Deltaps\Phi = 0$$
  where $\Deltaps$ is the reduced coproduct $\Deltaps(\Phi)= \Deltas(\Phi) - 1 \otimes \Phi - \Phi \otimes 1$. 
Equivalently, 
$\Phi_{w\sha w'} = 0$   for all words $w,w'\in \{e_0,e_1\}^{\times}$ of  total weight $N$  and total $\dd$-degree $r$.
\end{defn}

\subsubsection{Stuffle product}
Let $Y=\{ \y_n, n\geq 1\}$ denote an alphabet with one letter $\y_i$ in every degree $\geq 1$, and consider the graded  algebra
$\Q\langle Y\rangle$ equipped with the stuffle product \cite{R}. It is  defined recursively by
\begin{equation} \label{stuffprod}
\y_i w \stu \y_j w' = \y_i (w \stu \y_j w') + \y_j (\y_i w \stu w') + \y_{i+j} (w \stu w')
\end{equation} 
for all $w,w'\in Y^\times$ and $i,j \geq 1$, and the property that the empty word $1$ satisfies $1\stu w=w\stu 1=w$.  
A linear map $\Phi: \Q\langle Y \rangle\rightarrow \Q$ is a homomorphism for the stuffle multiplication,
or $\Phi_{w} \Phi_{w'} = \Phi_{w \stu w'}$ for all $w,w'\in Y^{\times}$ and $\Phi_1=1$,  if and only if  
$$\Phi=\sum_w \Phi_w w  \in \Q \langle \langle Y \rangle \rangle^{\times}$$
is group-like for the (completed) coproduct $\Deltat : \Q \langle \langle Y \rangle \rangle \rightarrow \Q \langle \langle Y \rangle \rangle \widehat{\otimes}_{\Q} \Q \langle \langle Y \rangle \rangle  $ which is 
a homomorphism for the concatenation product and  defined on generators by
\begin{equation}\label{Deltastuff}
\Deltat \y_n = \sum_{i=0}^n \y_i \otimes \y_{n-i}\ .\end{equation}
 Thus  the  stuffle relations are equivalent to being group-like for $\Deltat$:
\begin{equation} \label{stuffrel}
\Phi \hbox{ homomorphism  for  } \stu \quad \Longleftrightarrow  \quad \Phi   \in \Q \langle \langle Y \rangle \rangle^{\times} \hbox{ and } \Deltat \Phi= \Phi \widehat{\otimes} \Phi
\end{equation} 
One says that $\Phi$ satisfies the stuffle relations if either of the equivalent conditions $(\ref{stuffrel})$ holds. Passing  to the corresponding Lie algebra, we have an equivalence
\begin{equation} \label{stuffrelmodprod}
\Phi_{w\stu w'} = 0 \hbox{  for all } w,w' \quad \Longleftrightarrow  \quad \Phi   \in \Q \langle \langle  Y\rangle \rangle \hbox{ and } \Deltat \Phi= 1 \widehat{\otimes} \Phi + \Phi \widehat{\otimes} 1
\end{equation} 
One says that $\Phi$ satisfies the stuffle relations \emph{modulo products} if either of the equivalent conditions $(\ref{stuffrelmodprod})$ holds. 

The algebra $\Q\langle Y\rangle$ is  graded for the degree (where $\y_n$ has degree $n$), and filtered for the depth (where $\y_n$ has depth $1$). 
By inspection of $(\ref{stuffprod})$, we notice that the right-most term is of lower depth than the other terms and therefore drops out in the associated graded. It follows that 
\begin{equation} 
\gr_{\dd} ( \Q \langle Y \rangle, \stu) \cong (\Q \langle Y \rangle, \sha)
\end{equation} 
and the associated graded is simply the shuffle algebra on $Y$. The depth induces a decreasing filtration on the (dual) completed Hopf algebra $\Q\langle \langle Y \rangle \rangle$, and 
it follows from  $(\ref{Deltastuff})$ that the images of the elements $\y_n$ are primitive in the associated graded. Thus
\begin{equation}\label{grdeltastisdeltash}
\gr_{\dd} \Deltat = \Deltas
\end{equation}
where $\Deltas: \Q\langle \langle Y \rangle \rangle\rightarrow \Q\langle \langle Y \rangle \rangle \widehat{\otimes}_{\Q} \Q\langle \langle Y \rangle \rangle$ is the (completed) coproduct for which the elements $\y_n$ are primitive, and which is a homomorphism for  concatentation.

\begin{defn}  Let $\Phi \in T(\Q Y)$ of degree $N\geq 2$ and $\dd$-degree $r$.  It defines a linear map $w\mapsto \Phi_w$ on words  in $Y$ of  weight $N$ and $\dd$-degree $r$. We say that it 
 satisfies the \emph{linearized stuffle relations} if 
  $$\Deltaps \Phi = 0$$
  where $\Deltaps$ is the  reduced coproduct of $\Deltas$ for  which the  $\y_n$ are primitive.
Equivalently, 
$\Phi_{w\sha w'} = 0$   for all words $w,w'\in Y$ of  total weight $N$  and total $\dd$-degree $r$.
\end{defn}

\subsubsection{Linearized double shuffle relations}
In order to consider both relations simultaneously, define a map
$$\alpha:  \Q \langle e_0, e_1 \rangle\rightarrow \Q\langle Y \rangle  $$
which maps every word beginning in $e_0$ to $0$, and such that 
$$\alpha (e_1 e_0^{a_1}\ldots e_1 e_0^{a_r} ) = \y_{a_1+1}\ldots \y_{a_r+1}\ .$$
In \cite{R}, Racinet considered a certain graded vector space, denoted  $\DMD_0(\Q)$ (d\'efinition 2.4 in \cite{R}),  of series satisfying the shuffle and stuffle relations and a regularization condition, and showed that it is a Lie algebra for the Ihara bracket.
Since we consider the depth-graded version of this algebra, the regularization plays no role here.
\begin{defn}  \label{deflinearizedDSrel} Let $\Phi \in \gr^{\dd}_r  \U \g$ of weight $N$. The linear form $w\mapsto \Phi_w$ on words of  weight $N$ and $\dd$-degree $r$
 satisfies the \emph{linearized  double shuffle  relations} if 
  $$\Deltas \Phi = 0 \qquad \hbox{ and } \qquad \Deltat \alpha(\Phi) =0 \ .$$
When $r=1$, we add the extra condition that $\Phi=0$ when $N$ is even, and $\Phi(e_0)=\Phi(e_1)=0$. 
Let $\gdsh\subset \U \g$ denote the vector space of elements satisfying the linearized double shuffle relations.
It is bigraded by weight and depth.
\end{defn}

\begin{rem} There is a natural  inclusion
\begin{equation} \label{LStoDSH}  \gr_{\dd} \DMD_0(\Q)  \To \gdsh \ .
\end{equation}
The graded subspace of $\gdsh$ of weight $N$ and depth $d$ is isomorphic to the vector space denoted $D_{N+d,d}$ in \cite{IKZ}.  In \cite{IKZ} it is conjectured that
$(\ref{LStoDSH})$ is an isomorphism. 
\end{rem}

Racinet proved that $\DMD_0(\Q)$ is preserved by the Ihara bracket. Since the latter is homogeneous for the $\dd$-degree, it follows that
$\gr_{\dd} \DMD_0(\Q)$ is also preserved by  the bracket, but this is not quite enough to prove that $\gdsh$ is too. 

\begin{thm}
The bigraded vector space $\gdsh$ is preserved by  the Ihara bracket.
\end{thm}
\begin{proof}
The compatibility of the shuffle product with the Ihara bracket  follows by  \mbox{definition}. It therefore suffices 
to check that the linearized stuffle relation is preserved by the bracket.  The proof of \cite{R} goes through identically, and uses in an essential way the fact that the images of the elements $\y_{2n}$ 
in $\gr^1_{\dd} \DMD_0(\Q)$ are zero (which in $\DMD_0(\Q)$ follows from proposition 2.2 in \cite{R}, but holds in $\gdsh$  by definition \ref{deflinearizedDSrel}). 
\end{proof}

It would be interesting to know if a suitable linearized version of the associator relations is equivalent to the linearized double shuffle relations.

\subsection{Summary of definitions} We have the following Lie algebras:
$$ \gmzv \subseteq  \DMD_0(\Q) \ ,$$
where $\gmzv$ is the image of the (weight-graded) Lie algebra of $\UMT$  and is isomorphic to the free  graded Lie algebra on (non-canonical) generators $\sigma_{2i+1}$ for $i\geq 1$.  A standard conjecture states that this is an isomorphism.
Passing to the depth-graded versions, and writing  $\gd= \gr_{\dd}\gmzv$, we have inclusions of bigraded Lie algebras
\begin{equation}\label{thelatter}
\gd\subseteq  \gr_{\dd} \, \DMD_0(\Q)   \subseteq \gdsh\end{equation}
where $\gdsh$ stands for the linearized double shuffle algebra. Once again, all Lie algebras in $(\ref{thelatter})$ are conjectured to be equal. 
The  bigraded dual space  $(\gd)^{\vee}$ is isomorphic to the Lie coalgebra  of depth-graded motivic multiple zeta values, modulo $\zetam(2)$ and modulo products.
All the above Lie algebras can be viewed inside $\U \g= T(\Q e_0 \oplus \Q e_1)$,  which is  graded for the $\dd$-degree. 

Next we show that complicated expressions in the non-commutative algebra $\U \g$ can be greatly simplified by encoding words of fixed $\dd$-degree in terms of  polynomials.

\section{Polynomial representations} \label{sectPolyrep}

\subsection{Composition of polynomials } Recall from $(\ref{Ugdescription})$ that $\U \g$ is isomorphic to the bigraded $\Q$-tensor algebra on $e_0,e_1$, and the space 
$\gr_{\dd}^r \, \U \g$ of $\dd$-degree  $r$ is the subspace spanned by words in $e_0,e_1$ with exactly $r$ occurrences of $e_1$.

\begin{defn} Consider the isomorphism of  vector spaces 
\begin{eqnarray} \rho: \gr_{\dd}^r \, \U \g  & \To & \Q[y_0,\ldots, y_r]  \\
e_0^{a_0} e_1 e_0^{a_1} e_1 \ldots e_1 e_{0}^{a_{r}} & \mapsto & y_0^{a_0}y_1^{a_1}\ldots y_r^{a_r} \nonumber
\end{eqnarray}
It maps elements of degree $n$ to elements of  degree $n-r$.
\end{defn}

The operator $\circb: \U \g \otimes_{\Q} \U \g \rightarrow \U\g$  respects the $\dd$-grading, so clearly  defines a map
\begin{eqnarray} \circb: \Q[y_0,\ldots, y_r] \otimes_{\Q} \Q[y_0,\ldots, y_s]   & \To & \Q[y_0,\ldots, y_{r+s}]   \\
f(y_0,\ldots, y_r) \otimes g(y_0,\ldots, y_s)  & \mapsto &  f\circb g\, ( y_0,\ldots, y_{r+s}) \nonumber
\end{eqnarray}
which can be read off  from equation $(\ref{circbdef})$. Explicitly, it  is
\begin{multline}\label{circformula}
f \circb g \, (y_0,\ldots, y_{r+s})  =  \sum_{i=0}^s f(y_i,y_{i+1}, \ldots, y_{i+r}) g(y_0,\ldots, y_i, y_{i+r+1}, \ldots, y_{r+s}) + \\
  (-1)^{\deg f + r} \sum_{i=1}^s f(y_{i+r},\ldots, y_{i+1},y_i) g(y_0,\ldots, y_{i-1}, y_{i+r}, \ldots, y_{r+s}) 
 \end{multline}
Antisymmetrizing, we obtain a formula for the Ihara bracket
\begin{eqnarray} \label{Iharadef}
\g \wedge \g \To \g  \\
\{ f, g\} = f \circb g - g \circb f \nonumber \ .
\end{eqnarray}
The linearized double shuffle relations on words translate into functional equations for polynomials after applying the map $\rho$. We describe some of these below.

\subsection{Translation invariance} The additive group $\G_a$ acts on $\A^{r+1}$, and hence $\G_a(\Q)$ acts on its ring of functions by   translation
$\lambda: (y_0,\ldots, y_r) \mapsto (y_0+\lambda,\ldots, y_r+\lambda)$.
\begin{lem} \label{lemtrans} The image of  $\gmzv$ under $\rho$ is contained in the subspace of polynomials  in $\Q[y_0,\ldots, y_r]$ which are invariant under translation.
\end{lem}

\begin{proof} The algebra $\gmzv\subset \g \subset \U \g $ is contained in the subspace of elements which are primitive
with respect to the shuffle coproduct $\Delta_{\sha}$. Let $\pi_0: \U \g \rightarrow \Q$ denote the linear map which sends the word $e_0$ to $1$ and all other words to $0$, and consider the map
$\partial_0 = (\pi_0 \otimes id) \circ \Delta_{\sha}$. It defines a derivation
$\partial_0: \U \g \To \U \g$
which satisfies 
$$\partial_0 (e_0^{a_0} e_1 \ldots e_1 e_0^{a_r} )= \sum_{i=0}^r a_i  \, e_0^{a_0} e_1 \ldots  e_1 e_0^{a_i-1}e_1  \ldots  e_1 e_0^{a_r} $$
 for all non-negative integers $a_0,\ldots, a_r$. 
 Let $r\geq 1$,  $\xi \in \gr_{\dd}^r \U \g$, and $f=\rho(\xi)$. If $\xi$  is primitive for $\Delta_{\sha}$, it satisfies  $\partial_0 \xi=0$ and the previous equation is
 \begin{equation} \label{parsums} \sum_{i=0}^r { \partial f \over \partial y_i }=0
 \end{equation} 
 This is equivalent to ${\partial \over \partial \lambda} g=0$, where $g= f(y_0,\ldots, y_r) - f(y_0+\lambda,\ldots, y_r+\lambda)$. Therefore $g$ is constant in $\lambda$, and vanishes at $\lambda=0$, so  $g\equiv 0$ and  $f$ is translation invariant.   \end{proof}
  Let us denote the map which sends $y_0$ to zero  and $y_i$ to $x_i$ for $i=1,\ldots, r$ by 
\begin{eqnarray} \Q[y_0,\ldots, y_r] &\To & \Q[x_1,\ldots, x_r]   \\
f  &\mapsto &\overline{f}  \nonumber \end{eqnarray}
where $\overline{f}$ is the `reduced' polynomial  $\overline{f}(x_1,\ldots, x_r)= f(0,x_1,\ldots, x_r)$. In the case when  $f$ is translation invariant,  we can retrieve $f$ from $\overline{f}$ via
\begin{equation} \label{fvfbar}
f(y_0,\ldots, y_r)=\overline{f}(y_1-y_0,\ldots, y_r-y_0) 
\end{equation}
Taking the  coefficients of  $(\ref{fvfbar})$  gives  equation {\bf I2} of \cite{BrMTZ}, which gives a formula for the shuffle-regularization of iterated integrals
which begin with any sequence of $0$'s.

\begin{defn}
For any element $\xi \in \gd$ of depth $r$ we denote its \emph{reduced} polynomial representation by $\overline{\rho}(\xi)\in \Q[x_1,\ldots, x_r]$.\end{defn} 
 To avoid confusion, we reserve  the variables $x_1,\ldots, x_r$ for  the reduced polynomial  $\overline{\rho}(\xi)$
and use the variables $y_0, \ldots, y_r$  as above to denote the full polynomial $\rho(\xi)$. 

\subsection{Antipodal symmetries} The set of primitive elements in a Hopf algebra is stable under the antipode. For the shuffle Hopf algebra this is the map
$* : \U \g \rightarrow \U \g$ considered earlier. 
Restricting to  $\dd$-degree $r$ and transporting via $\rho$ we obtain a map
\begin{eqnarray}
\sigma: \Q[y_0,\ldots, y_r] & \To & \Q[y_0,\ldots, y_r]  \\
\sigma( f) (y_0,\ldots, y_r) & = & (-1)^{\deg(f) +r} f(y_r,\ldots, y_0)  \nonumber
\end{eqnarray}
  Therefore, if  $f\in \Q[y_0,\ldots, y_r]$ satisfies the shuffle relations $(\ref{shuffrelmodprod})$, then
\begin{equation} \label{fsf}
f+\sigma(f)=0\ .
\end{equation}
Since the  stuffle algebra, graded for the depth filtration, is isomorphic to the shuffle algebra on $Y$ $(\ref{grdeltastisdeltash})$, it follows that its antipode
is  the map $\y_{i_1}\ldots \y_{i_r} \mapsto (-1)^r\y_{i_r}\ldots \y_{i_1}$. This defines an involution
\begin{eqnarray}
\overline{\tau}: \Q[x_1,\ldots, x_r] & \To & \Q[x_1,\ldots, x_r]  \\
\overline{\tau}(\overline{f}) (x_1,\ldots, x_r) & = & (-1)^r \overline{f}(x_r,\ldots, x_1)  \nonumber
\end{eqnarray}
  Therefore if  $\overline{f}\in \Q[x_1,\ldots, x_r]$ satisfies the linearized stuffle relations, then
$\overline{f} + \overline{\tau}(\overline{f})=0$.
Note that  the involution $\overline{\tau}$ lifts to an involution
 \begin{eqnarray}
 \tau: \Q[y_0,y_1,\ldots, y_r] & \To & \Q[y_0, y_1,\ldots, y_r]  \\
\tau(f) (y_0,y_1,\ldots, y_r) & = & (-1)^r f(y_0,y_r,\ldots, y_1)  \nonumber
\end{eqnarray}
and therefore if  $f\in \Q[y_0,\ldots, y_r]$  satisfies both translational invariance  and the \mbox{linearized} stuffle relations, we have:
\begin{equation} \label{ftf}
f+ \tau(f)=0 \ .
\end{equation}
The composition $\tau \sigma$ is the signed cyclic rotation of order $r+1$: $$\tau \sigma(f) (y_0,\ldots, y_r) = (-1)^{\deg f}  f(y_r,y_0,\ldots, y_{r-1})$$
and plays an important role in the rest of this paper.

\subsection{Parity relations} The following result is well-known, and was first proved by Tsumura \cite{Tsumura}, and subsequently in (\cite{IKZ}, theorem 7). We repeat the proof  for convenience. It has subsequently been extended to multiple polylogarithms by Panzer \cite{Panzer}. 

\begin{prop} \label{propparity} The components of $\gdsh$  in weight $N$ and depth $r$ vanish unless $N\equiv r \mod 2$.  Equivalently,   $\rho(\gdsh)$  consists of polynomials of even degree only.
\end{prop}
\begin{proof}
Let $f\in \Q[y_0,\ldots, y_r]$ be in the image of $\rho(\gd_r)$. In particular it satisfies the linearized stuffle relations, $(\ref{fsf})$ and $(\ref{ftf})$. Following \cite{IKZ}, consider the relation
$$\y_1 \sha \y_2 \ldots \y_r = \y_1\ldots \y_r + \sum_{i=2}^{r} \y_2 \ldots \y_i  \y_1 \y_{i+1} \ldots \y_r \ ,$$
where $r\geq 2$ (the case $r=1$ follows from definition \ref{deflinearizedDSrel}).
 Then  we have
$$ f(y_0,y_1,\ldots, y_r) + \sum_{i=2}^{r} f(y_0,y_2,\ldots, y_i, y_1, y_{i+1}, \ldots, y_r) =0\ .$$
Now apply the automorphism of $\Q[y_0,\ldots, y_r]$ defined on generators by $y_i \mapsto y_{i+1}$, where $i$ is taken modulo $r+1$. This leads to the equation
$$ f(y_1,y_2,\ldots, y_r, y_0) + \sum_{i=3}^{r+1} f(y_1,y_3,\ldots, y_i, y_2, y_{i+1}, \ldots,y_r, y_0) =0\ .$$
By applying a cyclic rotation $\tau \sigma$ to each term,  we get
\begin{multline} f(y_0, y_1,y_2,\ldots, y_r) + \sum_{i=3}^{r} f(y_0, y_1,y_3,\ldots, y_i, y_2, y_{i+1}, \ldots, y_r)   \nonumber \\
+ f(y_2, y_1, y_3,\ldots, y_r,y_0)  =0\ .
\end{multline}
The first line can be interpreted as the terms occurring in the  linearized stuffle product ($\y_2 \sha \y_1\y_3 \ldots \y_r$) minus the first  term. As a result, one obtains  the equation
$$ -f(y_0, y_2,y_1, y_3,\ldots, y_r)+f(y_2, y_1, y_3,\ldots, y_r,y_0)  =0  \ ,$$
which, by a final application of $ \tau \sigma$ to the right-hand term, yields
$$ ((-1)^{\deg f} -1 )f(y_0, y_2,y_1, y_3,\ldots, y_r) =0 \ .$$
Therefore, in the case when $\deg f$ is odd, the polynomial $f$  must vanish.  
\end{proof}

\subsection{Dihedral symmetry and the Ihara bracket} \label{sectdihedihara}
For all $r\geq 1$, consider the  graded vector space $\PP_r$ of polynomials  $f\in \Q[y_0,\ldots, y_r]$ which satisfy 
\begin{eqnarray} \label{goodcyclicpolys}
f(y_0,\ldots, y_r) & =  &f(-y_0,\ldots,- y_r) \\ 
f + \sigma(f)  &= & f+   \tau (f) \quad  = \quad 0 \ .   \nonumber
\end{eqnarray}
The maps $\sigma,\tau$ generate a dihedral group $D_{r+1}=\langle \sigma, \tau \rangle $ of symmetries acting on $\PP_r$ of  order $2r+2$,  and 
any   element $f$ satisfying $(\ref{goodcyclicpolys})$ is invariant under cyclic rotation:
$$  f = \sigma \tau (f)  \quad \hbox{ where } \quad  (\sigma \tau)^{r+1} = id \ , $$ 
and anti-invariant under the  reflections $\sigma$ and $\tau$. Thus $\Q f \subset \PP_r$  is isomorphic to the one-dimensional sign (orientation) representation  $\varepsilon$ of the dihedral group $D_{r+1}$.

\begin{prop} Suppose that $f\in \PP_r$ and $g\in \PP_s$ are polynomials  satisfing $(\ref{goodcyclicpolys})$.  Then
the Ihara bracket is the signed average over the dihedral symmetry group:
\begin{eqnarray} \label{dihedbracket}
\{f,g\} = \sum_{\mu  \in D_{r+s+1}} \varepsilon(\mu)  \mu\big( f(y_0,y_1,\ldots, y_r) g(y_{r},y_{r+1}\ldots, y_{r+s})\big)
\end{eqnarray}
In particular,  $\{ . , .\}: \PP_r \times \PP_s \rightarrow \PP_{r+s}$, and $\PP=\bigoplus_{r\geq 1} \PP_r $ is a bigraded Lie algebra.
\end{prop}
\begin{proof}
A straightforward calculation from $(\ref{circformula})$ and definition $(\ref{Iharadef})$ gives
$$\{f,g\} = \sum_{i }  f(y_i, y_{i+1}, \ldots, y_{i+r}) \big( g(y_{i+r}, y_{i+r+1}, \ldots,y_{i-1})- g(y_{i+r+1}, y_{i+r+2}, \ldots,y_i)    \big)$$
where the summation indices are taken modulo $r+s+1$. The Jacobi identity for $\{., .\}$ is automatic since the Ihara action is an action, but can also be proved very easily by  identifying its terms  with the set of  double cuts in a polygon (see \cite{Anatomy}.)
\end{proof}
A similar dihedral symmetry was also found by Goncharov \cite{dihedral}; the interpretation of the dihedral reflections as antipodes may or may not be new.
Since the Ihara action is, by definition, compatible with the shuffle product, it follows from lemma \ref{lemtrans}  that translation invariance is preserved by the  Ihara bracket. One can also easily verify this by direct computation: 
$$\sum_{i=0}^r {\partial f \over \partial y_i }=\sum_{i=0}^s {\partial g \over \partial y_i }= 0  \quad \Longrightarrow \quad  \sum_{i=0}^{r+s} {\partial \{f,g\} \over \partial y_i } =0\ .$$
\begin{defn}
Let $\overline{\PP}_r \subset \PP_r $ denote the subspace of polynomials  which satisfy $(\ref{goodcyclicpolys})$ and are invariant under translation, and write $\overline{\PP}= \bigoplus_{r \geq 1} \overline{\PP}_r$. 
\end{defn}

By abuse of notation, we can equivalently view $\overline{\PP}_r$ as the subspace of polynomials  $\overline{f}\in \Q[x_1,\ldots, x_r]$ whose lift 
  $ \overline{f}(y_1 -y_0,\ldots, y_r-y_0)$ lies in $\PP_r$.  Explicitly,  $\overline{\PP}_r$ is the vector space of polynomials satisfying
 \begin{enumerate}
 \item  $\overline{f}(x_1,\ldots, x_r )=  \overline{f}(-x_1,\ldots, -x_r ) $
\item  $\overline{f}(x_1,\ldots, x_r ) + (-1)^r \overline{f}(x_r,\ldots, x_1 ) =0 $
\item $\overline{f}(x_1,\ldots, x_r ) + (-1)^r \overline{f}(x_{r-1}-x_r,\ldots, x_1-x_r, -x_r ) =0 $
\end{enumerate}
with Lie bracket is induced by $(\ref{dihedbracket})$. To conclude the previous discussion, the map 
$$\overline{\rho} : \gd  \To \overline{\PP}$$
is an injective map of bigraded Lie algebras.
\begin{defn} We use the notation $\dD_r\subset  \Q [x_1,\ldots, x_r]$ to denote the space $\overline{\rho}(\gdsh_r)$ in depth $r$.  It is the space of polynomial solutions to the linearized double shuffle equations in depth $r$
and is the direct sum for all $n$, of spaces  denoted $D_{n, r}$ in \cite{IKZ}.
\end{defn}

\subsection{Generators in depth $1$ and examples}
It follows from theorem \ref{DI} that in depth $1$,  the  Lie algebra $\gd$ has exactly one generator in every odd weight $\geq 3$:
$$\overline{\rho} (\gd_1 ) = \bigoplus_{n\geq 1} \Q \, x_1^{2n} $$
In particular, the algebras $\gd \subset \gr_{\dd} \DMD_0 (\Q) \subset \gdsh$ are all isomorphic in depth $1$.

\begin{defn}Denote the Lie subalgebra generated by  $x_1^{2n}$, for $n\geq 1$, by
\begin{equation} \label{gevdef} \godd \subset \gd\ .
\end{equation} 
\end{defn}
\begin{example} \label{exdepth2bracket} The formula for the Ihara bracket in depth 2 can be written:
$$
\{ x_1^{2m}, x_2^{2n}\} = x_1^{2m} \big((x_2-x_1)^{2n} - x_2^{2n}\big) + (x_2-x_1)^{2m} \big( x_2^{2n} - x_1^{2n}\big)+ x_2^{2m} \big( x_1^{2n} - (x_2-x_1)^{2n}\big)
$$
\end{example}

\section{Modular relations in depth two} \label{sectmod2}

\subsection{Reminders on period polynomials} We recall some definitions from (\cite{KZ}, \S1.1).
Let  $S_{2k}(\mathrm{SL}_2(\Z))$  denote the space of   cusp forms of weight $2k$  for the full modular group $\mathrm{SL}_2(\Z)$.

\begin{defn} \label{defperiodpoly}ÊLet $n\geq 1$ and let $W_{2n}^e\subset \Q[X,Y]$ denote the vector space of \mbox{homogeneous} polynomials $P(X,Y)$  of degree  $2n-2$ satisfying
\begin{eqnarray}
&P(X,Y) + P(Y,X)=  0 \quad , \quad  P(\pm X, \pm Y) = P (X, Y) & \label{periodsign} \\ 
&P(X, Y) + P(X-Y,X) + P (-Y, X-Y) =0 \ . &  \label{periodrel}
\end{eqnarray}
\end{defn}
The  space $W_{2n}^e$ contains the polynomial $p_{2n}= X^{2n-2}-Y^{2n-2}$, and is  a direct sum
$$W_{2n}^e \cong  \Ss_{2n}  \oplus \Q \,p_{2n} $$
where $\Ss_{2n}$ is the subspace of polynomials which vanish at $(X,Y)=(1,0)$.  We write $\Ss= \bigoplus_{n} \Ss_{2n}$. The Eichler-Shimura-Manin theorem  gives a map which associates, in particular, an even  period polynomial 
to every cusp form:
$$S_{2k}(\SL_2(\Z)) \To W_{2k}^e\otimes_{\Q}\C\ .$$
Explicitly, if $f$ is a cusp form of weight $2k$, the map is given by
\begin{equation}\label{EMmap}
 f \quad \mapsto \quad  \Big( \int_{0}^{i \infty} f(z) (X-z Y)^{2k-2} dz\Big)^{+}
 \end{equation}
where $+$ denotes the projection onto invariants under  the involution $(X,Y)\mapsto (-X,Y)$, i.e., $a^+(X,Y)= {1 \over 2} (a (X,Y) + a(-X,Y))$. The three-term equation $(\ref{periodrel})$ follows from integrating around the geodesic triangle with vertices
$0,1,i \infty$ and is reminiscent of the hexagon equation for associators.
The map $(\ref{EMmap})$ is  an isomorphism onto a  subspace of $W^e_{2k} \otimes \C$ of codimension $1$. Thus $\dim S_{2k}(\SL_2(\Z)) = \dim \Ss_{2k}$ and  it follows from classical results on the space of modular forms  of level one that:
\begin{equation} \label{cuspgf} \sum_{n\geq 0} \dim \Ss_{2n} s^{2n} = {s^{12} \over (1-s^4)(1-s^6)}  = \mathbb{S}(s)\ . \end{equation}

\subsection{Linearized double shuffle in depths two and three} We first recall from \cite{IKZ} that $\dD_2$ is the graded space  of
 homogeneous polynomials $f\in \Q[x_1,x_2]$ in two variables satisfying the linearized double shuffle equations in depth two:
$$f(x_1,x_2) + f(x_2,x_1) =0 \qquad \hbox{ and } \qquad f(x_1,x_1+x_2) + f(x_2, x_1+x_2) =0\ . $$
The space $\dD_3$ consists of  homogeneous polynomials $f \in \Q[x_1,x_2,x_3]$ such that
\begin{eqnarray} f(x_1,x_2,x_3) + f(x_2,x_1,x_3) +f(x_2,x_3,x_1)  &=& 0  \nonumber \\ 
f^{\sharp}(x_1,x_2,x_3) + f^{\sharp}(x_2,x_1,x_3) +f^{\sharp}(x_2,x_3,x_1)  &=& 0  \nonumber \ ,
\end{eqnarray} 
where $f^{\sharp}(x,y,z) = f(x,x+y, x+y+z)$. The general double shuffle equations, and their linearized versions are derived 
 in  \cite{Anatomy} \S\S 4-7  using Hopf-algebra  theoretic techniques.

\subsection{A short exact sequence in depth 2} The Ihara bracket gives a map
\begin{equation}
\{ . , . \} : \gdsh_1 \wedge \gdsh_1 \To \gdsh_2\ .
\end{equation}
Applying the isomorphism $\overline{\rho}$ leads to a map
\begin{equation} \label{D1wedgeD1}
  \dD_1 \wedge \dD_1 \To  \dD_2\ ,
\end{equation}
 given by  the formula in example \ref{exdepth2bracket}. 
Since $\dD_1$ is isomorphic to the graded vector space $\Q[x_1^{2n},{n\geq 1}]$, it follows that $\dD_1 \wedge \dD_1$ is isomorphic 
to the space of antisymmetric even polynomials  $p(x_1,x_2)$ with positive bidegrees, with  basis  $x_1^{2m} x_2^{2n} - x_1^{2n} x_2^{2m} $ for $m>n>0$.
It follows from example \ref{exdepth2bracket} that  the image of  $p(x_1,x_2)$ under $(\ref{D1wedgeD1})$ is
$$p(x_1,x_2) + p(x_2-x_1, -x_1)+p(-x_2, x_1-x_2)$$
Comparing with $(\ref{periodrel})$ and $(\ref{periodsign})$, we immediately deduce (c.f., \cite{IT, GG, GKZ, Sch}) that 
\begin{equation} 
\ker ( \{. , .  \} : \dD_1 \wedge \dD_1 \To \dD_2 )  \overset{\sim}{\To} \Ss\ .
\end{equation} 
In fact, the dimensions of the space $\dD_2$ have  been computed several times in the literature (for example, by some simple representation-theoretic arguments), and it is relatively easy to show \cite{GKZ} that the following sequence is exact:
\begin{equation} \label{depth2seq}
0 \To \Ss \To \dD_1 \wedge \dD_1 \To \dD_2 \To 0\ .
\end{equation}

\begin{example} The smallest 
non-trivial period polynomial occurs in degree 10   and is given by 
$s_{12} = X^2Y^2(X-Y)^3(X+Y)^3  = X^8 Y^2 -3 X^6 Y^4+3 X^4 Y^6-X^2Y^8.$ By the exact sequence $(\ref{depth2seq})$ it immediately 
gives rise to the equation
\begin{equation} \label{IharaRel}
3 \{x_1^4,x_1^6\} = \{x_1^2,x_1^8\} \ ,
\end{equation} 
which, by the faithfulness of the map $\overline{\rho}$ is equivalent to Ihara's formula \eqref{IntroIharaRel}.
\end{example}

\subsection{A short exact sequence in depth 3} If $V$ is a  vector space let $\Lie_n(V) \subset V^{\otimes n}$
denote the component of degree $n$ in the free Lie algebra $\Lie(V)$ viewed inside the tensor algebra $T(V)= \bigoplus_{n\geq 0} V^{\otimes n}$.  The triple Ihara bracket gives a map
$$\Lie_3 (\gdsh_1) \To \gdsh_3\ ,$$
and hence a map $\Lie_3(\dD_1) \rightarrow \dD_3$ whose image is
spanned by  $\{x_1^{2a} , \{ x_1^{2b}, x_1^{2c}\} \!\}$, for $a,b,c\geq 1$.
Goncharov has studied the space $\dD_3$, and computed its dimensions in each weight \cite{GG}. It follows from his work  that  the following sequence
\begin{equation} \label{depth3seq} 0 \To \Ss \otimes_{\Q} \dD_1 \To \Lie_3 (\dD_1)  \To \dD_3\To 0
\end{equation} 
is exact, where the first map (identifying $\Ss$ with $\ker (\Lambda^2 \dD_1 \rightarrow \dD_2)$) is given by
$$\Ss \otimes_{\Q} \dD_1 \hookrightarrow \Lambda^2 \dD_1 \otimes_{\Q} \dD_1 \rightarrow \Lie_3(\dD_1)\ ,$$
and  the second map in the sequence immediately above is  $[a,b]\otimes c \mapsto [c, [a,b]] $.
Starting from depth $4$, the structure of $\gdsh_d\cong \dD_d$ is not known\footnote{After writing the first version of this paper, S. Yasuda kindly sent me his private notes \cite{Y} in which he gives a conjectural
group-theoretic interpretation for the dimensions of $\dD_4$, in accordance with the Broadhurst-Kreimer conjecture.}. In particular, it is easy to show that the map given by the quadruple Ihara bracket
$$\Lie_4 (\dD_1) \To \dD_4$$
is not surjective, since in weight $12$, $\dim \dD_4= 1$, but $\Lie_4 (\dD_1) = 0$. Our next  purpose  is to construct the missing elements in depth 4.

\begin{rem} A different  way to think about the sequence $(\ref{depth3seq})$ is via the curious equality $\dim \Ss \otimes_{\Q} \dD_1 = \dim \Lambda^3(\dD_1)$
 which follows from $(\ref{cuspgf})$. I do not know if there is an  appropriate combinatorial or modular  interpretation  of this identity
 which could be  relevant to the previous exact sequence.
   \end{rem}

\section{Exceptional modular elements in depth four} \label{sectExcep}

\subsection{Linearized equations in depth four} For the convenience of the reader, we write out the linearized double shuffle relations in full in depth four.  There are four equations. In order to write them down we shall use the following notation, where $f \in \Q[x_1,\ldots, x_4]$,
and we are given any set of indices $\{i,j,k,l\} = \{1,2,3,4\}$:
\begin{eqnarray}
f(ijkl) & = & f(x_i,x_j,x_k,x_l) \\
f^{\sharp}(ijkl) &=  &f(x_i,x_i+x_j,x_i+x_j+x_k,x_i+x_j+x_k+x_l) \ .\nonumber 
\end{eqnarray}
Then $\dD_4$ (see \cite{IKZ}, \S8) is the subspace of polynomials $f\in \Q[x_1,\ldots, x_4]$ satisfying
\begin{eqnarray}Êf(1\sha 234)=0 \quad , \quad f(12\sha 34)=0  \label{d4linstuff}Ê\\
Êf^{\sharp}(1\sha 234)=0 \quad , \quad f^{\sharp}(12\sha 34)=0 \label{d4linshuff}Ê
 \end{eqnarray}
where $f$ and $f^{\sharp}$ are extended by linearity in the obvious way, and 
$$ 1 \sha 234 = 1234 + 2134 + 2314+ 2341$$
 $$ 12 \sha 34= 1234 + 1324 + 1342+ 3124 + 3142 + 3412$$
For example, the first equation of \eqref{d4linstuff} is simply
$$f(x_1,x_2,x_3,x_4) + f(x_2,x_1,x_3,x_4) + f(x_2,x_3,x_1,x_4) + f(x_2,x_3,x_4,x_1)=0\ .$$
 We construct some exceptional solutions to these equations from period polynomials.

\subsection{Definition of the exceptional elements} \label{sectdefnex}
Let $f(x,y) \in \Ss_{2n+2}$ be an even period polynomial of degree $2n$ which vanishes at $y=0$.   It follows from $(\ref{periodsign})$ and $(\ref{periodrel})$  that it vanishes along $x=0$ and $x-y=0$. Therefore we can write
$$f(x,y)=xy (x-y) f_0(x,y)$$  
where $f_0(x,y)\in \Q[x,y]$ is symmetric  and  homogeneous of degree $2n-3$, and  satisfies
\begin{equation} \label{f03term}
 f_0(x,y) + f_0(y-x,-x) + f_0( -y,x-y) =  0 \ .
\end{equation}
Let us also write $f_1(x,y)= (x-y) f_0(x,y)$.  We have $f_1(-x,y)=f_1(x,-y)=-f_1(x,y)$.

\begin{defn}\label{defnef}
Let $f\in \Q[x,y]$ be an even period polynomial as above. Define 
\begin{eqnarray} \label{eqnefdef}
\e_f &\in & \Q[y_0,y_1,y_2,y_3,y_4] \\
 \e_f & =  & \sum_{\Z/ 5 \Z } \Big(   f_1(y_4-y_3,y_2-y_1) + (y_0-y_1) f_0 (y_2-y_3,y_4-y_3) \Big)\ ,  \nonumber 
\end{eqnarray}
where the sum is over cyclic permutations
$(y_0,y_1,y_2,y_3, y_4 ) \mapsto (y_1,y_2,y_3, y_4, y_0)$.
Its reduction 
$\overline{\e}_f\in \Q[x_1,\ldots, x_4]$ is obtained by setting $y_0=0, y_i =x_i$, for $i=1,\ldots, 4$.
\end{defn}

\begin{rem} The full expression for  $\overline{e}_f$ is explicitly:
\begin{multline} \label{efexplicit} \overline{e}_f(x_1,x_2,x_3,x_4) = 
 f_1(x_4-x_3,x_2-x_1) + f_1(-x_4,x_3-x_2)    + f_1(x_1,x_4-x_3) \\ 
+ f_1(x_2-x_1,-x_4) +f_1(x_3-x_2,x_1)     
 -x_1  f_0 (x_2-x_3,x_4-x_3) +   (x_1-x_2) f_0 (x_3-x_4,-x_4) \\ + (x_2-x_3) f_0 (x_4, x_1)  + (x_3-x_4) f_0 (-x_1, x_2-x_1)    + x_4 f_0 (x_1-x_2, x_3-x_2) \ .
\end{multline} 
Since $f$ is even  it vanishes to order two along $x=0,y=0,x=y$. Therefore
$$f_0(0,y)  = f_0(x,0) = f_0(x,x)=0 $$
and the same holds a fortiori for  $f_1$.  If we set $x_3=x_4=0$ in the expression 
 \eqref{efexplicit} then all terms except for the fifth vanish and we are left with
 \begin{equation} \label{efat00}
\overline{\e}_f( x,y, 0,0) = f_1 (-y,x)= f_1(x,y) \ . 
\end{equation}
In this way,  the period polynomial $f$ can be retrieved from $\overline{\e}_f$:  it is $  xy\overline{\e}_f(x,y,0,0)$. 
This computation is  related to the discussion in \cite{DataMine}, \S9.2.
\end{rem}

\subsection{Exceptional elements and linearized double shuffle equations} 

\begin{thm} The  reduced polynomial  $\overline{\e}_f$  obtained from   $(\ref{eqnefdef})$ satisfies the linearized double shuffle relations. In particular, we have an injective linear map
$$\overline{\e}:\Ss  \To \dD_4 \ .$$
\end{thm}
\begin{proof} The injectivity follows immediately from $(\ref{efat00})$.  The proof that the linearized double shuffle relations hold is a finite computation. In the absence of a purely conceptual proof we shall break the calculation into   more easily verifiable pieces. 

We first consider the stuffle equations. It follows from general properties of the Dynkin operator on shuffle algebras that a homogeneous polynomial in  four variables satisfies the two linearized stuffle equations $(\ref{d4linstuff})$ if and only if it is in the image of the map  $\lambda:\Q[x_1,\ldots, x_4] \rightarrow \Q[x_1,\ldots, x_4]$, defined by
$$\lambda (f ) (x_1,\ldots, x_4) = \alpha (f)(x_1,x_2, x_3, x_4) -\alpha (f)(x_4,x_3, x_2, x_1)   $$ 
where $\alpha( f)(x_1,\ldots, x_4)$ is the linear combination 
$$ f(x_1,x_2,x_3,x_4) - f(x_1,x_2,x_4,x_3) - f(x_1,x_4, x_2, x_3) + f(x_1,x_4,x_3,x_2) \ .$$
For a detailed discussion and proofs, we refer to \cite{Anatomy}, \S16.4 and corollary 16.6. 
The  linearized stuffle relations $(\ref{d4linstuff})$ hold for the terms in $f_1$ and $f_0$ in \eqref{efexplicit} separately. 
Consider first the  terms in $f_1$. They consist of two parts:
$$ T_1= f_1(x_4-x_3,x_2-x_1)$$
and
$$ T_2= f_1(-x_4,x_3-x_2)    + f_1(x_1,x_4-x_3) + f_1(x_2-x_1,-x_4) +f_1(x_3-x_2,x_1) \ . $$
One easily  checks that $\lambda (T_1)  = 4 \, T_1$, and that  $ \lambda (f_1(x_1,x_2-x_3))$ equals 
$$  f_1(x_1,x_2-x_3) - f_1(x_1,x_2-x_4) - f_1(x_1,x_4-x_2) + f_1(x_1,x_4-x_3) = T_2 \ ,$$
using only the fact that $f_1$ is antisymmetric and odd in $x$ and $y$.  Thus $T_1, T_2$ lie in the image of $\lambda$ and are solutions to the linearised stuffle equations.

Now consider  the terms in $f_0$ in \eqref{efexplicit}. Once again, they break into two parts. 
$$T_3=  x_4 f_0 (x_1-x_2, x_3-x_2)   -x_1  f_0 (x_2-x_3,x_4-x_3)$$ 
and
$$T_4=   (x_1-x_2) f_0 (x_3-x_4,-x_4)  + (x_2-x_3) f_0 (x_4, x_1)  + (x_3-x_4) f_0 (-x_1, x_2-x_1)    \ .$$
One checks that $\lambda T_k = T_k$ for $k=3,4$ using the 3-term relation $(\ref{f03term})$, and hence  $T_3,T_4$ are  solutions to the linearized stuffle equations.  This is the only point in the proof where the 3-term relation is needed.

For the linearized shuffle relations,  note that there exists $g \in \Q[x,y]$ such that  $f_0(x,y) = (x+y) g(x,y)$ and $f_1 = (x^2-y^2) g(x,y)$ since  an even period polynomial $f(x,y)$ vanishes along $x=y$ and is even in both $x$ and $y$.  The polynomial $g(x,y)$ is symmetric  and odd in $x$ and in $y$, i.e.,  
 $g(x,y)=g(y,x)$ and $g(-x,y)=g(x,-y)=-g(x,y)$.  These properties suffice to prove that \eqref{efexplicit} satisfies the shuffle equations. The full expression again splits into two pieces with 4 and 6 terms respectively:
\begin{multline}  \nonumber T_5= (x_4^2-x_{23}^2) g(x_4, x_{23})+x_{23}(x_1+x_4)g(x_1, x_4)+\\
x_{12}(x_{34}-x_4)g(x_{34}, -x_4)+(x_{12}^2-x_4^2)g(x_{12}, x_4)\end{multline}
and 
\begin{multline} \nonumber T_6= x_{34}(x_{21}-x_1)g(x_1, x_{12})+x_4(x_{32}+x_{12})g(x_{12}, x_{32})+(x_{32}^2-x_1^2)g(x_{32}, x_1)\\
+(x_1^2-x_{43}^2)g(x_1, x_{43}) + x_1(x_{34}+x_{32})g(x_{32}, x_{34})+(x_{43}^2-x_{21}^2)g(x_{43}, x_{21}) 
\end{multline} 
where we use the notation $x_{ij} = x_i-x_j$. 
Both $T_5$ and $T_6$  can laboriously verified to satisfy the two shuffle equations in depth 4.
The statement follows since the exceptional element  $\e_f $ equals $ T_1+T_2+T_3+T_4 = T_5+T_6$.
\end{proof}
Identifying $\gdsh_4$ with $\dD_4$ via the map $\overline{\rho}$, we can view $\e$ as a  map from $\Ss$ to  $\gdsh_4$.
Note that the relation $(\ref{periodrel})$ is proved for the periods of modular forms by integrating round contours very similar to those which prove the symmetry and hexagonal relations for associators. It would be interesting to see if 
the five-fold symmetry  of the element $\e_f$ is related to the pentagon equation.

\begin{example}   \label{example12D4} It follows from $(\ref{cuspgf})$ that the space of period polynomials in degrees $12, 16, 18$ and $20$ is of dimension $1$. Choose integral generators:
\begin{eqnarray}
f_{12} & =  & [x_{{1}}^{8} , x_{{2}}^{2}]-3\,[x_{{1}}^{6},x_{{2}}^{4}] \nonumber \\
  f_{16} & =  &    2\,[x_{{1}}^{12},x_{{2}}^{2}]-7\,[x_{{1}}^{10},x_{{2}}^{4}]+11\,[x_{
{1}}^{8},x_{{2}}^{6}]\nonumber \\
 f_{18} & = & 8\, [x_{{1}}^{14},x_{{2}}^{2}]-25\,[x_{{1}}^{12},x_{{2}}^{4}]+26\,[x_
{{1}}^{10},x_{{2}}^{6}]
 \nonumber \\
  f_{20} & = &  3\,[x_{{1}}^{16},x_{{2}}^{2}]-10\,[x_{{1}}^{14},x_{{2}}^{4}]+14\,[x_
{{1}}^{12},x_{{2}}^{6}]-13\,[x_{{1}}^{10},x_{{2}}^{8}]
 \nonumber 
\end{eqnarray}
where $[x_1^a,x_2^b]$ denotes $x_1^ax_2^b-x_1^bx_2^a$.
 Let $\e_{12},\ldots, \e_{20}$ denote the corresponding exceptional elements.  We know by  theorem $\ref{DI}$ that $\gmzv$ is of dimension $2$ in weight twelve, spanned by $\{ \sigma_3, \sigma_9\}$
 and $\{\sigma_5, \sigma_7\}$. We know by $(\ref{IharaRel})$ that  in weight  twelve $\gd_2$ is of dimension one, $\gd_3$ vanishes by parity, so it follows that 
 $\gd_4$  is of dimension one and hence spanned by $\overline{\e}_{12}$ (since we know that $\gdsh_4$ in weight 12 is one-dimensional). 
Writing out just  a few of its  coefficients as an example, we have:
$$\overline{\e}_{12}=  x_3^7 x_4 -116\, x_1^3 x_2^2x_3^2x_4 -57\, x_1^2 x_2^5x_4+\ldots   \quad (118 \hbox{ terms  in total} )$$
Using $\overline{\e}_{12}$ one can write all depth-graded motivic multiple zeta values of depth four and weight twelve as multiples of $\zeta_{\dd}(1,1,8,2)$. For example, one has 
  $$  \zeta_{\dd}(4,3,3,2 ) \equiv -116\,  \zeta_{\dd}(1,1,8,2) \quad ,  \quad  \  \zeta_{\dd}(3,6,1,2 ) \equiv -57\,  \zeta_{\dd}(1,1,8,2) $$
modulo products and modulo multiple zeta values of depth $\leq 2$.
   
\end{example}

\subsection{Are the exceptional elements motivic?} \label{sectspectral} 
We say that an exceptional element $\e_f$ is \emph{motivic} if it lies in the depth-graded motivic Lie algebra:
$$\e_f  \quad \in  \quad \gd_4 \subseteq \gdsh_4\ . $$
\begin{conj} The exceptional elements $\e_f$ are all motivic. 
\end{conj}

Since  $\gd_k= \gdsh_k$ for $k =1,2,3$,  to prove that 
an exceptional element  $\e_f$ is motivic it is enough to show that, modulo commutators, it lies in the image  of the map
$$ d:  \Ss \To (\gd_4)^{ab}  $$
where $\Ss\subset \Lambda^2 \gd_1$ is the space of relations in depth $2$, and $d$ is the first non-trivial differential in the spectral sequence on $\gmzv$ associated to the depth filtration \S \ref{sect:DepthSpectralSequence}.

The map $d$ can be computed explicitly as follows. Choose a lift  $\widetilde{\sigma}_{2n+1}$ of every generator $\sigma_{2n+1} \in \gd_1$ to $\gmzv$, and decompose  it according to the $\dd$-degree:
$$\widetilde{\sigma}_{2n+1} =  \sigma^{(1)}_{2n+1} + \sigma^{(2)}_{2n+1} +  \sigma^{(3)}_{2n+1}+ \ldots  \ ,$$
where $\sigma^{(i)}_{2n+1}$ is of $\dd$-degree $i$, and $ \sigma^{(1)}_{2n+1} =\sigma_{2n+1}$.  Then for any element 
$$\xi = \sum_{i,j} \lambda_{ij}\,  \sigma_i \wedge \sigma_j  \quad \in \quad  \Ss = \ker (  \{. , . \}:  \Lambda^2 \gd_1\rightarrow \gd_2)$$
with $\lambda_{i,j} \in \Q$, we have
$$d \xi = \sum_{i,j} \lambda_{ij} ( \{\sigma^{(1)}_i, \sigma^{(3)}_j\} +\{\sigma^{(2)}_i, \sigma^{(2)}_j\} + \{\sigma^{(3)}_i, \sigma^{(1)}_j\} )\ ,
$$ 
where $\{. , . \}: \bigwedge^2 \gmzv \rightarrow \gmzv$ can be computed  on the level of polynomial representations 
by exactly the same formula given in \S \ref{sectPolyrep}. 

\begin{ex} \label{exEmot} The  elements $\widetilde{\sigma}_3, \widetilde{\sigma}_5, \widetilde{\sigma}_7, \widetilde{\sigma}_9$ defined by the coefficients of $\zeta(3),\zeta(5),\zeta(7)$, and $\zeta(9)$ in weights 3,5,7,9  in Drinfeld's associator are canonical, and we have
\begin{equation}  
\label{144equation}  \{ \widetilde{\sigma}_3, \widetilde{\sigma}_9 \} -3 \{ \widetilde{\sigma}_5, \widetilde{\sigma}_7 \}  = { 691 \over 144}Ê\,  \e_{12} \mod \hbox{ depth } \geq 5\ ,
\end{equation} 
which proves that the element $\e_{12}$ is motivic. Using the depth-parity  proposition \ref{propparity}, one can  show that  the corresponding  congruence
$$  \{ \widetilde{\sigma}_3, \widetilde{\sigma}_9 \} -3 \{ \widetilde{\sigma}_5, \widetilde{\sigma}_7 \}  \equiv  0  \mod 691\ ,$$
propagates to depth five also. 
Compare with the `key example' of \cite{YI}, page 258, and the ensuing discussion. Thereafter, one checks that
\begin{eqnarray}
 d \big(2  \, \sigma_3\wedge  \sigma_{13}  -7  \, \sigma_5 \wedge  \sigma_{11}   + 11 \, \sigma_7\wedge \sigma_{9}  \big)  & \equiv &  {3617 \over 720} \e_{16}  \pmod{ \mathfrak{a}} \nonumber \\
 d \big(8 \,  \sigma_3 \wedge \sigma_{15}  -25 \, \sigma_5 \wedge \sigma_{13}   + 26 \, \sigma_7 \wedge \sigma_{11}  \big)  & \equiv &  {43867 \over 9000} \e_{18} \pmod{ \mathfrak{a}} \nonumber \\
 d \big(3   \, \sigma_3 \wedge \sigma_{17}  -10 \, \sigma_5 \wedge \sigma_{13}   + 14 \, \sigma_7 \wedge \sigma_{13}  -13 \, f_9 \wedge f_{11} \big)  & \equiv &  {174611 \over 35280} \e_{20}  \pmod{ \mathfrak{a}} \nonumber 
\end{eqnarray}
where 
$\mathfrak{a} = \{\gmzv,\gmzv\} + \dd^5 \gmzv$, i.e., the previous identities hold 
modulo commutators and modulo terms of depth 5 or more.
In this manner, I  have checked that the elements $\e_f$ are motivic for all $f$ up to weight $30$.
In particular, it seems that the differential $d$ is related to our map $\e$  (which is defined over $\Z$) up to a non-trivial isomorphism of the space of period polynomials. The \mbox{numerators} on the right-hand side are the numerators of $\zeta(16)\pi^{-16}, \zeta(18)\pi^{-18}$, and $\zeta(20)\pi^{-20}$. 
Unfortunately,  it does not seem possible  to construct canonical zeta elements  $\widetilde{\sigma}_{2n+1}$ for $n\geq 5$ in a consistent way such that the above relations hold exactly in $\gmzv_4$ (and not modulo $\mathfrak{a}$). 
 \end{ex}
 
 Using the theory of the unipotent fundamental group of the Tate elliptic curve, we showed in \cite{BrZeta} how to construct canonical 
 elements  $\sigma^{(3)}_{2i+1}$ modulo depths $\geq 5$, which enables one to write down the differential $d$ explicitly. 
 The only remaining difficulty in proving that the elements 
 $\e_f$ are motivic is therefore to  understand better the quotient $\gdsh/ \{ \gdsh, \gdsh\}$ in depth $4$. 
Yasuda has since shown,  assuming that the $\e_f$ are motivic, how to relate the exceptional elements to the differential $d$ using the action of Hecke operators on the space of period polynomials.

 If the elements $\e_f$ can be shown to be motivic, then they provide in particular  an answer to the question raised by Ihara in (\cite{YI}, end of \S4 page 259).
  The  appearance of the numerators  of Bernoulli numbers  is related to conjecture 2 in \cite{YI} and  has been studied from the Galois-theoretic side by  Sharifi \cite{Sha} and McCallum and Sharifi \cite{MS}.

\section{Some properties of  the elements $\e_f$} \label{sectEfProp}  The exceptional elements $\e_f$ satisfy many remarkable properties,  only some of which will be  outlined  here. 
Of particular relevance are those properties which are \emph{motivic}, i.e., stable under the Ihara bracket.

\subsection{Unevenness} 
For any polynomial $f\in \Q[x_1,\ldots, x_r]$, let 
$$\pi^k_{x_i} f = \hbox{ coefficient of }Êx_i^k \hbox{ in } f$$
and denote the projection onto the even part  in $x_i$ by:
$$\pi^{ev}_{x_i} \,  f = \sum_{k\geq 0} (\pi^{2k}_i f)  x_i^{2k} \ .$$
We can also write  $\pi^{ev}_{x_i} f = \frac{1}{2} \left(  f(x_1,\ldots, x_i, \ldots, x_r) +  f(x_1,\ldots, -x_i, \ldots, x_r) \right)$.

\begin{lem} The elements $\e_f(y_0, y_1, y_2, y_3, y_4)$ are \emph{uneven}:
\begin{equation} \label{efuneven}
\pi_{y_0}^{ev} \pi_{y_1}^{ev} \pi_{y_2}^{ev} \pi_{y_3}^{ev} \pi_{y_4}^{ev} (\e_f )=0\ .
\end{equation}
\end{lem} 
\begin{proof} The term of the form $(y_0-y_1) f_0 (y_2-y_3,y_4-y_3)$ in 
definition \eqref{defnef}  is obviously uneven since it is linear in $y_0,y_1$. 
The term 
$ f_1(y_4-y_3, y_2-y_1)$
is likewise uneven because $f_1(x,y)$ is odd in $x$ and $y$.
The fact that $\e_f$ is uneven follows by cyclic symmetry.
\end{proof}
We shall see later  in \S \ref{sectEnumodd} that the property  of being uneven is motivic, i.e., is stable under  the Ihara bracket, and is related to the totally odd zeta values.
We conjecture that  a solution to the  linearised double shuffle equations  is uneven if and only if it is in the Lie ideal  of $\gdsh$ generated by exceptional elements.

\subsection{Sparsity}

\begin{lem}The elements $\e_f(y_0,y_1,y_2,y_3,y_4)$ are \emph{sparse}:
\begin{equation} \label{efsparse}
{\partial^5 \over \partial y_0 \partial y_1 \partial y_2 \partial y_3 \partial y_4} (\e_f )=0\ .
\end{equation}
In other words,  $\e_f$ is a linear combination of monomials which only depends on four out of five of the variables $y_0,\ldots, y_4$. 
\end{lem} 

\begin{proof} Immediate from the definition.
\end{proof}

One can show that this property is also motivic, i.e., forms an ideal under  the Ihara bracket.   Call an element $f\in \p_r$ sparse if it is 
annihilated by  $${\partial^{r+1} \over \partial y_0 \ldots \partial y_r}\ .$$

\begin{prop} The sparse elements in $\bigoplus_{r} \p_r$  form a Lie  ideal. 
\end{prop} 

\begin{proof}
Let $f \in \p_r$ and $g\in \p_s$ where $f$ is sparse. The Ihara bracket $\{f,g\}$ is given by  \eqref{dihedbracket}  and consists of a cyclic sum over  terms of the form 
$$\left( f(y_0,\ldots, y_r)  -  f(y_{r+s}, y_0, \ldots, y_{r-1})\right) g(y_r,\ldots, y_{r+s}) \ ,$$
to which we apply the product of $ \frac{\partial}{\partial y_i}$ for $0\leq i \leq r+s$. Using the chain rule with respect to $\frac{\partial}{\partial y_r}$ and 
$\frac{\partial}{\partial y_{r+s}}$
and invoking
the  sparsity of $f$, we obtain 
\begin{equation} \label{inproofsparse}  \left(  \prod_{i=0}^{r-1}   \frac{\partial}{\partial y_i}\right)  \left(  f(y_0,\ldots, y_r)  -  f(y_{r+s}, y_0, \ldots, y_{r-1})\right)  \times  \left(  \prod_{i=r}^{r+s}   \frac{\partial}{\partial y_i}  \right)  g(y_r,\ldots, y_{r+s})\ .
\end{equation}
By cyclic symmetry of $f$ we have 
$$  f(y_{r+s}, y_0, \ldots, y_{r-1}) =f(y_0,\ldots, y_{r-1}, y_{r+s})\ .$$
Using the sparsity of $f$  once more  we find that
 $$\left(  \prod_{i=0}^{r-1}   \frac{\partial}{\partial y_i}\right)    f(y_0,\ldots, y_r)   = \left(  \prod_{i=0}^{r-1}   \frac{\partial}{\partial y_i}\right)  f( y_0, \ldots, y_{r-1},y_{r+s}) $$
since the left-hand side is a polynomial which is  annihilated  by $\frac{\partial}{\partial y_r}$,  and hence does not depend on $y_r$. In particular, it equals the right-hand side.  It follows that the left-hand factor of 
\eqref{inproofsparse} vanishes, which completes the proof. 
\end{proof}

In fact, there are  other differential equations satisfied by the $\e_f$ and one can use these equations to define 
various filtrations on the Lie algebras $\p$ and $\gdsh$. It would be interesting to try to prove that the degree in the exceptional elements defines a grading 
on the Lie subalgebra of $\overline{\p}$ spanned by the $x_1^{2n}$ and the $\overline{\e}_f$, as predicted by the conjectures below. 

\begin{rem}The  properties of unevenness and sparsity  imply  that almost all the coefficients of $\e_f$ are zero. This  implies that the freeness of the motivic Lie algebra (theorem \ref{DI})   hangs by a thread (see, for example,  equation \eqref{144equation}).  

 If we define the \emph{interior} of a polynomial $p\in \Q[x_1,\ldots, x_r]$ to be
$p^{o} = \pi^{\geq 2}_1 \ldots \pi^{\geq 2}_r p$, then the majority of the non-trivial monomials in $\overline{\e}_f$ are determined by:
\begin{eqnarray} \label{efint}
(\overline{\e}_f)^{o} = f_1(x_4-x_3,x_2-x_1)^{o}
\end{eqnarray}
which follows easily from the definition \eqref{efexplicit}.
\end{rem} 
\subsection{Other properties} 
We mention briefly some directions for further investigation. The proofs of the following facts are trivial yet sometimes lengthy applications of the definitions, basic properties of period polynomials, and the definition of $\{., .\}$.  

 Suppose that $f^{(1)}, \ldots, f^{(n)}$ are period polynomials, and let $f^{(i)}_1$ be as defined  in \S \ref{sectdefnex}. Then, generalizing $(\ref{efat00})$,  we have
\begin{equation}
 \pi^0_{n+2}\pi^0_{n+3} \ldots\pi^0_{4n-1} \pi^0_{4n} \big(\overline{\e}_{f^{(1)}} \circ (\overline{\e}_{f^{(2)}} \circ \ldots (\overline{\e}_{f^{(n-1)}} \circ \overline{\e}_{f^{(n)}}) \cdots )\big) = \prod_{i=1}^n f_1^{(i)}(x_i,x_{i+1}) 
\end{equation}

Unfortunately, some information about the polynomials $f_i$ is lost in this equation,  but  one can   do better  by using the operators $\pi_i^{ev} = \pi_{x_i}^{ev}$. For example,  one checks that:
\begin{equation} \label{quadfactor}
\pi^{ev}_1 \pi^{ev}_5 \pi^0_6 \pi^0_7 \pi^0_8\,  (\overline{\e}_{f} \circ  \overline{\e}_g )=   \big(\pi^{ev}_1 \pi_4^0 \overline{\e}_f (x_1,x_2,x_3,x_4)\big)\times \big(  \pi^{ev}_5 \pi_6^0 \overline{\e}_g(x_3,x_4,x_5,x_6) \big)
 \end{equation}
 factorizes.
Applying the operator $\pi^2_3$ to this equation gives
$$ \big( \pi^1_3  \pi^{ev}_1\pi_4^0 \overline{\e}_f (x_1,x_2,x_3,x_4)\big)\times \big(  \pi^1_3 \pi^{ev}_5 \pi_6^0 \overline{\e}_g(x_3,x_4,x_5,x_6) \big)Ê\in \Q[x_1,x_2] \otimes_{\Q} \Q[x_4,x_5] $$ 
and  causes the variables to separate. Next, one  checks that
$$  \pi^1_3  \pi^{ev}_1\pi_4^0 \overline{\e}_f (x_1,x_2,x_3,x_4)  =  \pi^{ev}_1 ( \alpha (x_2 -x_1)^{\deg f} + f_0(x_1,x_1+x_2))$$
for some $\alpha \in \Q$, and it is easy to show that the right-hand side of the previous equation is non-zero and  uniquely determines the period polynomial $f$ (using the fact that the involutions $(x_1,x_2)\mapsto (x_1,x_2-2x_1)$ and $(x_1,x_2)\mapsto (-x_1,x_2)$ generate an infinite group).
Putting these facts together shows:
\begin{cor} 
There are no non-trivial  relations between  commutators 
$\{\overline{\e}_f, \overline{\e}_g\}$. 
\end{cor} 
 Since similar factorization properties as  $(\ref{quadfactor})$ hold in higher depths, one might hope to prove,  in a similar manner, that the Lie subalgebra of $\gdsh$ 
generated by the exceptional elements $\e_f$ is free.

\section{Lie algebra structure and Broadhurst-Kreimer conjecture} \label{sectinterp}

\subsection{Interpretation of the Broadhurst-Kreimer  conjecture}
In the light of  the Broadhurst-Kreimer conjecture on the dimensions of the space of multiple zeta values graded by depth $(\ref{introBK})$, and Zagier's conjecture
which states that the regularised double shuffle relations generate all relations between multiple zeta values, 
 it is natural to  rephrase their conjectures, tentatively, in the Lie algebra setting as follows:
 \begin{conj}  \label{conjmain} (Strong Broadhurst-Kreimer and Zagier conjecture)
\begin{eqnarray} \label{homologyconj}
H_1(\gdsh, \Q) & \cong & \gdsh_1 \oplus \e(\Ss)  \\
H_2(\gdsh, \Q) & \cong & \Ss \nonumber \\
H_i(\gdsh, \Q) & =  & 0  \quad \hbox{for all} \quad  i \geq 3\ .\nonumber 
\end{eqnarray}
\end{conj}
This  conjecture is the strongest possible conjecture that one could make: as we shall see below, it implies nearly 
all the remaining open  problems in the field. The numerical evidence for this conjecture is substantial \cite{DataMine}, but not sufficient to remove all reasonable doubt.
More conservatively, and without reference to the double shuffle relations, one could make  a weaker reformulation of the Broadhurst-Kreimer conjecture:
\begin{conj}  \label{conjsecond} (Motivic version of the Broadhurst-Kreimer conjecture). The exceptional elements
$\e_f$ are motivic (i.e., $\e(\Ss) \subset \gd$), and 
\begin{eqnarray} \label{homologyconj2}
H_1(\gd, \Q) & \cong & \gd_1 \oplus \e(\Ss)  \\
H_2(\gd, \Q) & \cong & \Ss \nonumber \\
H_i(\gd, \Q) & =  & 0  \quad \hbox{for all} \quad  i \geq 3\ .\nonumber 
\end{eqnarray}
\end{conj}
Since the conjectural generators  are totally explicit, it is possible
to verify  the independence of Lie brackets in the reduced polynomial representation $\overline{\rho}(\gmzv)$ simply by computing the coefficients of a  small number  of monomials. 
In this way, it should be possible to verify $(\ref{homologyconj2})$   to  much higher  weights and depths than is presently known. Note that the Broadhurst-Kreimer conjecture could fail
if there existed non-trivial relations between commutators involving several exceptional elements $\e_f$. These would necessarily have weight and depth far beyond the range of present computations.

\subsection{Enumeration of dimensions} \label{sectEnum} Let $\g$ be a  Lie algebra over a field $k$, and let  $\U \g$ be its universal enveloping algebra.  The homology groups $H_i(\g;k)$  of $\g$ are defined to be the homology of the following complex
\begin{equation} \label{CE2}
\cdots\To \Lambda^3 \g\To   \Lambda^2 \g \To   \g  \To 0
\end{equation}
Suppose that $\g$ is bigraded, and finite-dimensional in each bigraded piece. Then $\Lambda^i \g, \U \g$  inherit a bigrading too. For any bigraded $\g$-module $M$, which is finite dimensional in every bidegree,   define its Poincar\'e-Hilbert series by
$$ \mathcal{X}_{M} (s,t) = \sum_{m,n\geq 0}Ê \dim_k (M_{m,n})  s^m t^n\ .$$
Similarly, for a family of such modules $M^l$, $l\geq 0$, let us write:  
$$ \mathcal{X}_{M^{\bullet}} (r,s,t) = \sum_{l,m,n\geq 0}Ê \dim_k (M^l_{m,n}) r^ls^m t^n\ .$$
The following proposition follows from standard homological algebra. 
\begin{prop} \label{propdimUgHg} With the above assumptions, the Poincar\'e series of $\U \g$ and the homology of $\g$ are related by the equation:
\begin{equation} \label{Xformula} 
\mathcal{X}_{\U \g} (s,t) = {1 \over \mathcal{X}_{H_{\bullet}(\g)}(-1,s,t) } \ .
\end{equation}
\end{prop} 
\begin{proof} 
Recall that the Chevalley-Eilenberg complex (\cite{W}, \S7.7): 
\begin{equation} \label{CE}
\cdots \To \U \g \otimes_k \Lambda^2 \g \To \U \g \otimes_k   \g \To \U \g \overset{\varepsilon}{\To} k \To 0
\end{equation}
is  exact in all degrees, and hence  defines  a resolution of $k$.  Let $\varepsilon: \U\g \rightarrow k$ be the augmentation map.  Viewing $k$ as a $\U \g$-module via $\varepsilon$, and  applying the functor $M\mapsto  k\otimes_{\U \g} M $ to $(\ref{CE})$ gives the complex \eqref{CE2}. 

Writing $\Lambda^1 \g = \g$, and $\Lambda^0 \g =k$, the exactness of   $(\ref{CE})$ yields
$$1= \sum_{l\geq 0} (-1)^l  \mathcal{X}_{\U \g}(s,t) \mathcal{X}_{\Lambda^l \g}(s,t) =  \mathcal{X}_{\U \g}(s,t) \mathcal{X}_{\Lambda^{\bullet} \g}(-1,s,t) \ .   $$
Since  $(\ref{CE2})$ computes the homology of $\g$, we deduce that $$\mathcal{X}_{\Lambda^{\bullet} \g}(-1,s,t) =\mathcal{X}_{H_{\bullet}( \g)}(-1,s,t)$$ which
 implies the formula \eqref{Xformula}.
\end{proof}
\subsection{Corollaries of conjectures \ref{conjmain} and \ref{conjsecond}}
Let us first apply  proposition \ref{propdimUgHg} to the algebra $\gdsh$, bigraded by weight and depth. 
\begin{lem} Conjecture  \ref{conjmain} implies that
\begin{equation}  \label{Xformfinal}
\mathcal{X}_{\U \gdsh}(s,t) =  { 1 \over 1-  \mathbb{O}(s)\, t +  \mathbb{S}(s) \, t^2 - \mathbb{S}(s) \, t^4 } \ .
\end{equation}
If we identify  $\gdsh_d$ via the isomorphism $\overline{\rho}$ with the space of  polynomials $\dD_d$ satisfying the linearized double shuffle relations, 
 we obtain the   conjecture   stated in (\cite{IKZ}, appendix).
\end{lem} 

\begin{proof} 
Assuming conjecture \ref{conjmain},  we have:
\begin{eqnarray}
\mathcal{X}_{H_1(\gdsh)}(s,t) & = & \mathbb{O}(s)\, t + \mathbb{S}(s) \, t^4 \nonumber \\ 
 \mathcal{X}_{H_2(\gdsh)} (s,t) & = & \mathbb{S}(s) \, t^2\ , \nonumber 
\end{eqnarray}
where $\mathbb{O}$ and $\mathbb{S}$ were defined in $(\ref{EOSdef})$. Apply  $(\ref{Xformula})$ to conclude. 
\end{proof}

\begin{prop} Conjecture \ref{conjmain} is equivalent to:
$$\hbox{Conjecture } \ref{conjsecond}  \quad \hbox{ together with } \quad  \gd =  \gdsh\ .$$
\end{prop}

\begin{proof} 
The inclusion $\gd \subset \gdsh$ implies that for all weights $N$ and depths $d$, 
\begin{equation} \label{dimltdim} \dim_{\Q} (\U \gmzv)_{N,d} \leq \dim_{\Q} (\U \gdsh)_{N,d}\ .\end{equation}
This  uses the fact that $\gr_{\dd} \U \gmzv \cong \U \gd$, 
which follows from the Poincar\'e-Birkoff-Witt theorem.
Now we know from  theorem $\ref{DI}$  that 
$${1 \over 1- \mathbb{O}(s)} = \sum_{N\geq 0}  \big( \sum_{d\geq 0} \dim_{\Q}  (\U \gmzv)_{N,d} \big) s^N\ . $$
If conjecture   \ref{conjmain} holds, then    
specializing $(\ref{Xformfinal})$ to $t=1$ we obtain 
$${1 \over 1- \mathbb{O}(s)} = \sum_{N\geq 0}  \big( \sum_{d\geq 0} \dim_{\Q}  (\U \gdsh)_{N,d} \big) s^N \ ,$$
and therefore for all $N$, 
$$ \sum_{d\geq 0} \dim_{\Q}  (\U \gmzv)_{N,d}  = \sum_{d\geq 0} \dim_{\Q}  (\U \gdsh)_{N,d}  \ .$$
Since the dimensions in $(\ref{dimltdim})$ are non-negative, this implies equality in $(\ref{dimltdim})$, and
so  $\gr_{\dd} \U \gmzv \cong \U \gd = \U \gdsh$ and hence   $\gd = \gdsh$.  Replacing $\gd$ with $\gdsh$ in conjecture  \ref{conjmain}  gives the statement of conjecture \ref{conjsecond}, and proves the first direction of the implication. The converse is obvious. 
\end{proof} 

The conjecture $\gd=\gdsh$  implies that $\gmzv =  \DMD_0(\Q)$, which is  the statement that all relations between motivic multiple zeta values are generated by the  regularized double shuffle relations (which is equivalent
to a conjecture of Zagier's).
By Furusho's theorem \cite{F},  it would in turn  imply  Drinfeld's conjecture that all relations between  (motivic) multiple zeta values are generated by the associator relations.

\begin{cor} Conjecture \ref{conjsecond} implies a Broadhurst-Kreimer conjecture for motivic multiple zeta values. More precisely,  conjecture \ref{conjsecond} implies that
\begin{equation} \label{motBK}
\sum_{N, d\geq 0} (\dim_{\Q} \gr^{\dd}_d \Ho_{N}) \,  s^N t^d = { 1  + \mathbb{E}(s) t \over 1- \mathbb{O}(s) t + \mathbb{S}(s) t^2 - \mathbb{S}(s) t^4}\ ,
\end{equation}
where $ \Ho_{N}$ is the $\Q$-vector space generated by motivic multiple zeta values of weight $N$. 
\end{cor}
\begin{proof} Apply proposition \ref{propdimUgHg} to $\gd$. Then conjecture \ref{conjsecond} implies via $(\ref{Xformula})$ that
\begin{equation} \label{dimgmzv}  \sum_{N,d\geq 0} (\dim_{\Q}  \gr_{\dd}^d \U \gmzv_{N}) s^N t^d={ 1 \over 1-  \mathbb{O}(s)\, t +  \mathbb{S}(s) \, t^2 - \mathbb{S}(s) \, t^4 }  \ .\end{equation}
Equation $(\ref{motBK})$ follows from $(\ref{DHoisDAtensZeta2})$, which  implies that 
 $$\gr^{\dd} \Ho \cong \left(\gr^{\dd}  \Ao\right) \otimes_{\Q}   \left( \gr^{\dd} \Q[\zetam(2)] \right)  \ . $$
The statement follows from  $\gr^{\dd} \Ao \cong   (\gr_{\dd}\U \gmzv)^{\vee}$, and the fact that   
$1  + \mathbb{E}(s) t$  is
the Poincar\'e series for the bigraded algeba $$ \gr^{\dd} \Q[\zetam(2)] = \Q \oplus \bigoplus_{n\geq1} \zetam_{\dd}(2n)
\Q\  $$ 
using the motivic version of Euler's theorem: $\Q \zetam(2n) =  \Q \zetam(2)^n$   \cite{BrMTZ}.
\end{proof}

\section{Totally odd  multiple zeta values} \label{secttotallyodd} 

Let $\Ho^{odd}\subset \Ho$ denote the vector subspace generated by the elements
\begin{equation}\label{zetaodds}
\zetam (2n_1+1,\ldots, 2n_r+1) \ , 
\end{equation}
where $n_1,\ldots, n_r$ are integers $\geq 1$. Then 
$\gr^{\dd}Ê\Ho^{odd} \subset  \gr^{\dd} \Ho$ is the vector subspace spanned by the depth-graded versions  
$\zetam_{\dd}(2n_1+1,\ldots, 2n_r+1) $
of $(\ref{zetaodds})$.   It is clear from the linearized stuffle product formula that 
$\gr^{\dd} \Ho^{odd}$ is an algebra, indeed, it is a quotient of the shuffle algebra
$$\big( \Q\langle \3,\5, \7, \ldots, \rangle, \sha \big)$$
with exactly one generator $\mathsf{2n+1}$ in each degree $2n+1$, for $n\geq 1$. 
Let $\Ao^{odd}$ denote  $\Ho^{odd}$ modulo the  ideal generated by  $\zetam(2)$, and let 
$\gr^{\dd} \Ao^{odd}$ denote  $\gr^{\dd} \Ho^{odd}$ modulo  the ideal generated by $\zetam_{\dd}(2)$.
\begin{prop}  The space  $\Ho^{odd}$ is almost stable under the motivic coaction:
$$\Delta^{\Mot} ({\dd}_r \Ho^{odd})  \subseteq  \Ao\otimes {\dd}_r \Ho^{odd}  + \Ao \otimes {\dd}_{r-2} \Ho  \ .$$
Furthermore, the  group $\UMT$ acts trivially on the associated graded $\gr^{\dd} \Ho^{odd}$.
\end{prop}
\begin{proof} By the remarks at the end of \S \ref{sectIharacoact}, it suffices to compute the infinitesimal coaction $(\ref{mainformula})$ in odd degrees only. Therefore apply the operator $D_{2s+1}$ to the element
$$\Imot(0; 1 , \underbrace{0,\ldots, 0}_{2 n_1} , 1,\underbrace{0,\ldots, 0}_{2 n_2} , \ldots, 1,  \underbrace{0,\ldots, 0}_{2 n_r} ;1)$$
We use the terminology `subsequence' to denote a term which occurs on the left hand side of \eqref{mainformula}, and `quotient sequence' to denote a term which occurs on the right.
Every subsequence with two or more $1$'s gives rise to a quotient sequence of depth $\leq r-2$. 
Every subsequence of depth exactly 1 and of odd length is either of the form:
$$\Imot(0; \underbrace{0,\ldots, 0}_{odd}, 1 , \underbrace{0,\ldots, 0}_{odd};1) \quad \hbox{ or } \quad  \Imot(1; \underbrace{0,\ldots, 0}_{odd}, 1 , \underbrace{0,\ldots, 0}_{odd};0) $$
(which cannot occur since every pair of  successive $1$'s in the original sequence are separated by an even number of $0$'s), or is of the form:
$$\Imot(0; \underbrace{0,\ldots, 0}_{even}, 1 , \underbrace{0,\ldots, 0}_{even};1) \quad \hbox{ or } \quad  \Imot(1; \underbrace{0,\ldots, 0}_{even}, 1 , \underbrace{0,\ldots, 0}_{even};0) \ .$$
In this case, the quotient sequence  has the property that every pair of successive $1$'s are separated by an even number of $0$'s, which  defines 
an element of $\Ho^{odd}$.   In the case when the subsequence has no $1$'s, the left-hand side  of  $(\ref{mainformula})$ is zero and the action is trivial, which proves the last statement.
\end{proof}
 It follows  immediately from the proposition that the  action of the graded Lie algebra $\Lie^{gr} \UMT$ on the two-step quotients ${\dd}_r \Ho^{odd}/ {\dd}_{r-2} \Ho^{odd}$ factors through its abelianization 
 $\Lie^{gr} (\UMT)^{ab}$, which has canonical generators  in every odd degree $2r+1$, for $r\geq 1$.  
 Thus for every integer $n\geq 1$, there is a well-defined derivation
$$\partial_{2n+1} :  \gr^{\dd}_r \Ho^{odd}  \To  \gr^{\dd}_{r-1} \Ho^{odd} \ ,$$
which corresponds to the action of the canonical generator  $\sigma_{2n+1} \in \Lie^{gr} (\UMT)^{ab}$.
If $m_1+\ldots +m_r = n_1+\ldots +n_r$  are integers $\geq 1$, then we obtain numbers 
$$c_{\binom{m_1 \ldots m_r}{n_1 \ldots n_r}}= \partial_{2m_1+1}\ldots  \partial_{2m_{r-1}+1}  \partial_{2m_r+1}\,   \zetam_{\dd}(2n_1+1,\ldots, 2n_r+1) \in \Z
$$
where, by duality, 
$$c_{\binom{m_1 \ldots m_r}{n_1 \ldots n_r}} = \hbox{coefficient of } x_1^{2n_1} \ldots x_r^{2n_r}\quad  \hbox{  in } \quad   x_1^{2m_r}\circb (x_2^{2m_{r-1}} \circb ( \cdots (   x_{{r-1}}^{2m_{2}}\circb  x_{{r}}^{2m_{1}}) \! \cdots \! )\ .$$ 
Recall that the  action $\circb$ is given by the formula:
\begin{eqnarray} \Q[x^2_1] \otimes_{\Q} \Q[x_1,\ldots, x_{r-1}] & \To & \Q[x_1,\ldots, x_r] \nonumber \\ 
x_1^{2n} \circb g(x_1,\ldots, x_{r-1} ) & = & \sum_{i=1}^r \big( (x_i\! - \!  x_{i-1})^{2n} - (x_i \!  -\!  x_{i+1})^{2n}\big) g(x_1,\ldots, \widehat{x_i}, \ldots, x_r)  \nonumber 
\end{eqnarray}
where $x_0=0$ and $x_{r+1}=x_r$ (i.e., the term $(x_r-x_{r+1})^{2n}$ is discarded).
Note that $\circb$ coincides with $\circ$ here by proposition  \ref{propcircscompared}  since $x_1^{2n}$ lies in the image of $\g$.
 
If $S_{N,r}$ denotes the set of compositions of an integer $N$ as a sum of $r$ positive integers, let $C_{N,r}$ denote the 
$|S_{N,r}| \times |S_{N,r}| $ square matrix whose entries are the integers 
\begin{equation} \label{CNrdef} 
(C_{N,r} )_{i,j}= c_{\binom{s_i}{s_j}}\ , \quad s_i,s_j \in S_{N,r} \ .
\end{equation} 

\subsection{Enumeration of totally odd depth-graded multiple zeta values}\label{sectEnumodd}

\begin{defn} We say that a polynomial $f \in \Q[y_0, y_1,\ldots, y_r]$ is \emph{uneven} if the coefficient of 
$y_0^{2n_0} \ldots y_r^{2n_r} $
in $f$ vanishes for all $n_0,\ldots, n_r \geq 0$.
\end{defn}
Recall from $(\ref{efuneven})$ that the exceptional elements $e_f \in \gdsh$ are uneven. The following proposition is a kind of  dual  to the previous one.
\begin{prop}  \label{propuneven} The set of uneven elements in $\gdsh$ is an ideal for the Ihara bracket.
 \end{prop}

\begin{proof}  Let $f, g\in \rho(\gdsh)$ such that $f$ is uneven. It suffices to show that $\{f,g\}$ is uneven. 
By the parity result (proposition \ref{propparity}), we know that $f$ and $g$ are of even degree.   It follows from $(\ref{circformula}$) that
$\{f,g\}$ is a linear combination  of terms of the form
$$ f( y_{\alpha} ) g( y_{\beta})$$
where $\alpha, \beta$ are sets of indices with $|\alpha \cap \beta| = 1$. Since the polynomial $f$ is homogeneous of  even degree, it follows that the coefficient of 
$y_0^{2n_0} \ldots y_N^{2n_N}$ in $\{f,g\}$ is a  linear combination of  the coefficients of totally even monomials in $f$, which all vanish. In more detail, let us consider  two sets of indices with $\alpha \cap \beta = \{\gamma\}$, and the corresponding monomials which occur in $f,g$ respectively:
 $$c^f_{\alpha} \prod_{a\in \alpha} y_a^{m_a} \qquad \hbox{ and } \qquad  c^g_{\beta} \prod_{b\in \beta} y_b^{n_a} $$
 where $c^f_{\alpha}, c^g_{\beta} \in \Q$. 
 A  monomial   in $\{f,g\}$ is of the form 
$$ c^f_{\alpha}  c^g_{\beta}  \, y_{\gamma}^{m_{\gamma}+n_{\gamma}}\prod_{a \in \alpha \backslash \{\gamma\}} y_a^{m_a}  \prod_{ b \in \beta \backslash \{\gamma\}} y_b^{n_b} \ .$$
Suppose that  $m_{\gamma}+n_{\gamma}$ and all $m_a, n_b$ are even for $a \in \alpha \backslash \{\gamma\}$ and $b \in \beta \backslash \{\gamma\}$. 
If $m_{\gamma}$ is odd, then  $c^f_{\alpha}=0$ since $f$ is of even degree by depth-parity. In the case when $m_{\gamma}$ is even, $c^f_{\alpha}= 0$ since $f$ is uneven.  Therefore $\{f,g\}$ is uneven. 
\end{proof} 

\begin{cor} The Lie ideal in $\gdsh$ generated by the exceptional elements is orthogonal to $\gr^{\dd}Ê\Ao^{odd}$.
\end{cor} 
\begin{proof} The elements $\e_f$ are uneven. 
\end{proof} 

 In the light of conjecture \ref{conjsecond} it is therefore natural to expect that 
$\gr^{\dd} \Ao^{odd}$ is  dual to  $\g^{odd}$, which suggests the following:

\begin{conj} \label{conjoddzetas} (`Uneven' part of motivic Broadhurst-Kreimer conjecture)
\begin{equation}  \label{conjeqnmotoddzetas}
 \sum_{N\geq 0, d\geq 0} ( \dim_{\Q} \gr^{\dd}_d \Ao^{odd}_{N}) s^N t^d  = {1 \over 1- \mathbb{O}(s) t + \mathbb{S}(s) t^2} \ .
 \end{equation}
\end{conj}

Since the action of the operators $\partial_{2n+1}$  on the totally odd depth-graded motivic multiple zeta values can be computed explicitly in terms of binomial coefficients, one can hope to prove a version of this conjecture by elementary  methods. 
Indeed, assuming conjecture \ref{conjsecond}, the left-hand side of   $(\ref{conjeqnmotoddzetas})$ is the generating series $$   \sum_{N\geq 0, d\geq 0} \mathrm{rank\, } (C_{N,d})\, s^N t^d\ , $$ where $C_{N,r}$ are the matrices of binomial coefficients defined in $(\ref{CNrdef})$. Therefore,   one is led to   conjecture that this generating series 
also coincides with the right-hand  side of $(\ref{conjeqnmotoddzetas})$.
  I  have  verified this  up to weight 30.

Standard transcendence conjectures for multiple zeta values would then have it that 
if  $\mathcal{Z}^{odd}_{N,d}$ denotes the space of totally odd depth-graded multiple zeta values  modulo $\zeta(2)$,  of weight $N$ and  depth $d$, then  we obtain the  new conjecture:
\begin{equation}
 \sum_{N\geq 0, d\geq 0} \left( \dim_{\Q}  \mathcal{Z}^{odd}_{N,d} \right) s^N t^d  = {1 \over 1- \mathbb{O}(s) t + \mathbb{S}(s) t^2}\ .
\end{equation}
\begin{rem} These conjectures measure the relations between totally odd (motivic) multiple zeta values  \emph{modulo all (motivic) multiple zeta values of lower depth}, not just modulo totally odd (motivic) multiple zeta values of lower depth (in other words, $\mathcal{Z}^{odd}_{N,d}$ denotes the span of the totally odd zetas in the space of depth graded multiple zeta values, and not the depth graded of the  space of totally odd multiple zeta values). 
\end{rem}

\bibliographystyle{plain}
\bibliography{main}

\end{document}